\documentclass[11pt]{article}

  \usepackage{amsthm}
  \usepackage{amssymb}
  \usepackage{amsmath}
  \usepackage{mathrsfs}
  \usepackage[all]{xy}
  \usepackage{bbm}

\newcommand{\cB}{\mathcal{B}}
\newcommand{\cC}{\mathcal{C}}

\newcommand{\cK}{\mathcal{K}}

\newcommand{\cP}{\mathcal{P}}

\newcommand{\cV}{\mathcal{V}}

\newcommand{\rB}{\mathrm{B}}
\newcommand{\rC}{\mathrm{C}}

\newcommand{\rH}{\mathrm{H}}

\newcommand{\rU}{\mathrm{U}}

\newcommand{\rZ}{\mathrm{Z}}

\newcommand{\sP}{\mathscr{P}}

\newcommand{\sU}{\mathscr{U}}

\newcommand{\Hom}{\mathrm{Hom}}

\def\si{\sigma}

\newcommand{\Si}{\Sigma}

\newcommand{\ga}{\gamma}

\newcommand{\io}{\iota}

\newcommand{\ute}{\underline{\theta}}

\newcommand{\sst}[1]{\scriptscriptstyle{#1}}
\author{John E. Roberts$^{(1)}$, Giuseppe Ruzzi$^{(1)}$, 
        Ezio Vasselli$^{(2)}$ \\[5pt] 
\small{$^{(1)}$ Dipartimento di Matematica, Universit\`a di 
       Roma ``Tor Vergata'' }\\
      \small{Via della Ricerca Scientifica I-00133, Roma,  Italy} \\[5pt] 
\small{$^{(2)}$ Dipartimento di Matematica, Universit\`a di 
       Roma ``La Sapienza'' }\\
      \small{Piazzale Aldo Moro 5, I-00185 Roma, Italy}\\[5pt] 
    \small{\texttt{roberts@mat.uniroma2.it}},\\ 
            \small{\texttt{ruzzi@mat.uniroma2.it}}, \\ 
            \small{\texttt{vasselli@mat.uniroma2.it}}}
\title{A theory of bundles over posets}
\date{}
 
\begin{document}

  \maketitle


  \theoremstyle{plain}
  \newtheorem{definition}{Definition}[section]
  \newtheorem{theorem}[definition]{Theorem}
  \newtheorem{proposition}[definition]{Proposition}
  \newtheorem{corollary}[definition]{Corollary}
  \newtheorem{lemma}[definition]{Lemma}

  \theoremstyle{definition}
  \newtheorem{remark}[definition]{Remark}
  \newtheorem{example}[definition]{Example}

\theoremstyle{definition}
  \newtheorem{ass}{\underline{\textit{Assumption}}}[section]

\section{Abstract} 

In algebraic quantum field theory the spacetime manifold is replaced 
by a suitable base for its topology ordered under inclusion. We 
explain how certain topological invariants of the manifold can be 
computed in terms of the base poset. We develop a theory of 
connections and curvature for bundles over posets in search of 
a formulation of gauge theories in algebraic quantum field theory.



\section{Introduction} 

This paper is motivated by Quantum Field Theory. Whilst the pseudo-Rie-
mannian spacetime manifold enters directly into the formulation of  
Classical Field Theory, it enters into Quantum Field Theory only 
through a suitable base for its topology ordered under inclusion. 
This raises the question as to what extent the topological data of the 
spacetime manifold are still encoded in the partially ordered set. 
Positive answers have been given as far as connectivity \cite{Rob03} and the 
fundamental group go \cite{Rob03,Ruz05} and we shall see 
in this paper that the 
same applies to the first locally constant cohomology group.\smallskip 

Perhaps the most important open question in Quantum Field Theory 
concerns gauge theories. There is no formulation of gauge theories 
which goes beyond a perturbative framework. Whilst it is too much to 
hope for a single rigorous example of a gauge theory in the near 
future, one might hope that certain of their structural features could 
be axiomatized so as to be able to predict aspects of their behaviour.
In particular, one might hope to say something about their 
superselection structure which remains a problem even for 
Quantum Electrodynamics.\smallskip 

Classical gauge theories are formulated in the language of differential 
geometry: there are principal bundles, associated bundles, connections 
and curvature, all referring to the underlying spacetime manifold. The 
very least change to be made to adapt the formalism to the quantum case 
would be to use the analogues of the structures of differential geometry 
over the poset derived from a base for the topology of the spacetime 
manifold. These analogues form the subject of this paper.\smallskip

Here we now provide an outline of the paper. \smallskip 

In Section \ref{X} we describe a symmetric simplicial set
$\tilde \Si_*(K)$ associated with a poset $K$. In particular we
observe that the nerve $\Si_*(K)$ is a subsimplicial set and that 
the fundamental groupoid  of $\tilde \Si_*(K)$ is isomorphic to that
of $\tilde \Si_*(K^\circ)$ 
associated with the opposite poset $K^\circ$.\smallskip 

In Section \ref{A} we introduce a class of bundles over posets: 
the \emph{net bundles}. A net bundle $\cB$ over a poset $K$ yields 
a fibration of symmetric simplicial sets 
$\pi_*:\tilde \Si_*(B)\to \tilde \Si_*(K)$ with a \emph{net structure} 
$J$: a 1--cocycle of the nerve $\Si_*(K)$ taking values in the groupoid 
of bijections between the fibres. 
The net structure is used in place of continuity (or
differentiability) and morphisms, 
cross sections and connections should be compatible with the net 
structure. A connection, in particular, is defined as an 
extension of the net structure to the simplicial 
set $\tilde \Si_*(K)$.  An important feature is that any net bundle 
admits one, and only one, 
\emph{flat connection}: the unique 1--cocycle $\tilde \Si_*(K)$
extending the net structure $J$ (Proposition \ref{Ab:3}). 
In accordance with these ideas, principal net bundles are defined in 
Section \ref{B} by adding a suitable action of a group on 
the total space. \smallskip

The analysis of principal net bundles 
relies on the observation that   
local representatives of connections, morphisms and
cross sections are all \emph{locally constant}. 
So, in Section \ref{C}, we exploit  this feature to  provide  an equivalent 
description of principal net bundles and their connections  in terms of the
cohomology of posets as described in \cite{RR06}:
the category of 
principal net bundles of a poset $K$ having  structure group $G$ is
equivalent to the category 
of 1--cocycles of $K$ taking values in $G$
(and analogously for connections, see Theorem \ref{Cb:2}).
Afterwards,  in Section \ref{D}, we transport  the notion of curvature, 
the relation between flatness and homotopy,  the existence of 
non-flat connections, and the Ambrose-Singer theorem
from the cohomology of posets to principal net bundles.\smallskip

In Section \ref{E}, we introduce the \v Cech cohomology of a poset.
Then, we  specialize our discussion to the 
poset of open, contractible subsets of a manifold. In this case, 
the above constructions yield 
the locally constant cohomology of the manifold,
which, as is well known, describes the category of flat
bundles (Theorem \ref{thm_hmor} and Proposition
\ref{prop_c_lcc}).\smallskip

We conclude the paper with an appendix recalling briefly the results
of  \cite{RR06}.


\section{Homotopy of posets}
\label{X}
In this section we analyze the
simplicial sets associated with 
a poset. We start by introducing 
symmetric simplicial sets and defining their fundamental groupoid.
Afterwards, we consider symmetric simplicial sets associated with 
categories and establish a relation between the fundamental groupoid 
of a category and that of the  corresponding opposite category.
Then we specialize to posets. We conclude with some remarks 
on the topology of partially ordered sets.  
References for this section  are \cite{May, RR06}.

\subsection{Simplicial Sets} 
\label{Xa}
Our investigation of the relation between invariants of a topological space 
and those of a suitable base for its topology ordered under inclusion makes 
substantial use of simplicial sets. A simplicial set is a contravariant 
functor from the simplicial category $\Delta^+$ to the category of sets. 
$\Delta^+$ is a subcategory of the category of sets having as objects 
$n:=\{0,1,\dots,n-1\}$, $n\in\mathbb N$ and as mappings the order 
preserving mappings. A simplicial set has a well known description 
in terms of generators, the face and degeneracy maps, and relations. 
We use the standard notation $\partial_i$ and $\sigma_j$ for the face 
and degeneracy maps, and denote the compositions
$\partial_{i}\partial_{j}$, $\si_{i}\si_{j}$, respectively, by 
$\partial_{ij}$, $\si_{ij}$.
A path in a simplicial set is an expression 
of the form $p:=b_n*b_{n-1}*\cdots*b_1$ where the $b_i$ are $1$--simplices 
and $\partial_0b_i=\partial_1b_{i+1}$ for $i=1,2,\dots,n-1$. We set 
$\partial_1p:=\partial_1b_1$ and $\partial_0p:=\partial_0b_n$. Concatenation 
gives us an obvious associative composition law for paths and in this way 
we get a category without units.\smallskip 

Homotopy provides us with an equivalence relation on this structure. This 
is the equivalence relation generated by saying that two paths of the form 
$p=b_n*b_{n-1}*\cdots*b_i*\partial_1c*b_{i-1}*\cdots*b_1$ and 
$q=b_n*b_{n-1}*\cdots*b_i*\partial_0c*\partial_2c*b_{i-1}*\cdots*b_1$, 
where $c$ is a $2$--simplex, are equivalent, $p\sim q$. Quotienting by 
this equivalence relation yields the homotopy category of the simplicial 
set.\smallskip 

We shall mainly use symmetric simplicial sets. These are contravariant functors 
from $\Delta^s$ to the category of sets, where $\Delta^s$ is the full subcategory 
of the category of sets with the same objects as $\Delta^+$. A symmetric 
simplicial set also has a description in terms of generators and relations, 
 where the generators now include the permutations of adjacent vertices, 
denoted $\tau_i$. In a symmetric simplicial set we define the reverse of a 
$1$--simplex $b$ to be the $1$--simplex $\overline{b}:=\tau_0b$ 
and the reverse of a path $p=b_n*b_{n-1}*\cdots*b_1$ is the path 
$\overline{p}:=\tau_0b_1*\tau_0b_2*\cdots*\tau_0b_n$. The reverse acts as 
an inverse after taking equivalence classes so the homotopy category 
becomes a homotopy groupoid.\smallskip 

We shall be concerned here with three different simplicial sets that can be 
associated with a poset and we give their definitions not just for a poset 
but for an arbitrary category $\cC$. The first denoted $\Sigma_*(\cC)$ is 
just the usual nerve of the category. Thus the $0$--simplices are just the 
objects of $\cC$, the $1$--simplices are the arrows of $\cC$ and a 
$2$--simplex $c$ is made up of its three faces which are arrows satisfying 
$\partial_0c\partial_2c=\partial_1c$. The explicit form of higher simplices 
will not be needed in this paper. The homotopy category of $\Sigma_*(\cC)$ 
is canonically isomorphic to $\cC$ itself. 

The second simplicial set is just the nerve $\Sigma_*(\hat\cC)$ of the 
category of fractions $\hat\cC$ of $\cC$ \cite{GZ}. 
The proof of the result in this section requires 
some knowledge of how $\hat\cC$ is constructed. Consider $\cC$ and the 
opposite category $\cC^{\circ}$. These categories have the same set of 
objects so we may consider paths 
$p:=b_n*b_{n-1}*\cdots*b_1$ where $\partial_0b_i=\partial_1b_{i+1}$, 
$i=1,2,\dots,n-1$ as before but now the $b_i$ can be taken at liberty to 
be arrows of $\cC$ or $\cC^{\circ}$. We now take the equivalence relation 
generated by homotopy within adjacent arrows of $\cC$, homotopy 
within adjacent arrows of $\cC^{\circ}$ and two further relations 
$p*b*b^{-1}*q\sim p*q$ where $b^{-1}$ denotes the arrow corresponding 
to $b$ in the corresponding opposite category and where $p$ and $q$ are 
not both the empty path. Finally $b*b^{-1}\sim 1_{\partial_0b}$. Quotienting 
by this equivalence relation yields the category of fractions $\hat\cC$. The 
arrows of $\hat\cC$ can be written in normal form.  The units of $\hat\cC$ 
on their own are in normal form. The other terms in normal form involve 
alternate compositions of  arrows of $\cC$ and of $\cC^{\circ}$ but no units 
nor compositions of an arrow and its inverse. Every arrow of $\hat\cC$ is 
invertible so that $\Sigma_*(\hat\cC)$ is in a natural way a symmetric 
simplicial set.\smallskip 

The third simplicial set $\tilde\Sigma_*(\cC)$ is also symmetric and is 
constructed as follows. Consider the poset $P_n$ of non-void subsets 
of $\{0,1,\dots,n-1\}$ ordered under inclusion. Any mapping $f$ from 
$\{0,1,\dots,m-1\}$ to $\{0,1,\dots,n-1\}$ induces an order preserving 
mapping from $P_m$ to $P_n$. Regarding the $P_n$ as categories, 
we have realized $\Delta^s$ as a subcategory of the category of 
categories. We then get a symmetric simplicial set where an 
$n$--simplex of $\tilde\Sigma_*(\cC)$ is a functor from $P_n$ to $\cC$.
  A $1$--simplex of $\tilde\Sigma_*(\cC)$ is a pair $(b_0,b_1)$ 
    of arrows of $\cC$ with $\partial_0b_0=\partial_0b_1$ and 
    $\partial_0(b_0,b_1)=\partial_1b_0$, 
    $\partial_1(b_0,b_1)=\partial_1b_1$. A $2$--simplex of 
    $\tilde\Sigma_*(\cC)$ is a set 
    $$(c_0,c_1,c_2;c_{01},c_{02},c_{10},c_{12},c_{20},c_{21})$$ 
    of nine arrows of $\cC$ with
    $$c_0c_{00}=c_1c_{10},\quad c_0c_{01}=c_2c_{20},\quad 
    c_1c_{11}=c_2c_{20}.$$
    The faces of this $2$--simplex $c$ are given by 
    $\partial_0c=(c_{20},c_{10})$, $\partial_1c=(c_{21},c_{01})$, 
    $\partial_2c=(c_{12},c_{02})$.
    An explicit description of the higher simplices will not be 
    needed.
\begin{theorem}
\label{Xa:1}
The homotopy groupoids of $\Sigma_*(\hat\cC)$ 
and $\tilde\Sigma_*(\cC)$ are isomorphic.
\end{theorem}
\begin{proof} 
Given a path $p:=b_0*b_1*\cdot*b_n$ in $\tilde\Sigma_1(\cC)$ 
    we define a map $\phi$ into paths in $\Sigma_1(\hat\cC)$, setting     
    $\phi(p):=\phi(b_0)\phi(b_1)\cdots\phi(b_n)$, 
    where for $b:=(b_0,b_1)$, 
    $\phi(b):=b_0^{-1}b_1\in\Sigma_1(\hat\cC)$. 
    If $p\sim p'$ then using the barycentric decomposition 
    inherent in a $2$--simplex $c$ of $\tilde\Sigma_*(\cC)$, we see 
    that 
\[
c_{00}^{-1}c_{01}c_{20}^{-1}c_{21}=c_{10}^{-1}c_{11}.
\]
    Hence $\phi(p)\sim\phi(p')$. Conversely, given 
    a path $\hat p\in\Sigma_1(\hat\cC)$,
\[
\hat p=\hat b_0*\hat b_1*\cdots *\hat b_n,
\] 
 each $\hat b_i$ being an arrow in $\hat\cC$, we set 
    $\psi(\hat p)=\psi(\hat b_0)*\psi(\hat b_1)*\cdots*\psi(\hat b_n)$    
    where $\psi(\hat b_i)$ is defined using the normal form of $\hat b_i$. 
  
 If the normal form of $\hat b$ is $b_0b_1\cdots b_m$ then 
    $\psi(\hat b)=\psi(b_0)*\psi(b_1)*\cdots*\psi(b_m)$ where 
    $\psi(b):=(\sigma_0\partial_0b,b)$ for $b$ an arrow of $\cC$ and 
    $\psi(b^{-1}):=(b,\sigma_0\partial_0b)$ when $b^{-1}$ is an 
    arrow of $\cC^{\circ}$. For a unit we have 
    $\psi(\sigma_0a)=(\sigma_0a,\sigma_0a)$. We claim that 
    $\hat p\sim\hat p'$ implies $\psi(\hat p)\sim\psi(\hat p')$. 
    It suffices to show that 
    $\psi(\hat p)\sim\psi(\hat b_0\hat b_1\cdots\hat b_n)$ 
    or that 
    $\psi(b_0)*\cdots*\psi(b_k)\sim\psi(b_0b_1\cdots b_k)$.
    Thus the problem has been reduced to showing that 
    the homotopy class of the left hand side does not 
    change as $b_0b_1\dots b_k$ is put into normal 
    form. This can be done using the following moves: 
    composing two adjacent arrows of $\cC$, 
    composing two adjacent arrows of $\cC^{\circ}$, 
    removing an arrow and its inverse if they are 
    adjacent unless they are the only arrows in 
    which case they are to be replaced by the identity. 
    In the first case it is enough to note that 
    $$(\sigma_0\partial_0b,b)*(\sigma_0\partial_0b',b')
    \sim(\sigma_0\partial_0b,bb')$$ 
    as follows by taking a $2$--simplex with 
    $c_0=c_1=\sigma_0\partial_0b$ and $c_3=b$. 
    The second case follows by analogy. For the 
    third case we need to know that 
    $$(b,\sigma_0\partial_0b)*(\sigma_0\partial_0b,b)
    \sim(\sigma_0\partial_1b,\sigma_0\partial_1b)$$ 
    as follows by taking a $2$--simplex $c$ with 
    $c_0=c_2=\sigma_0\partial_0b$ and $c_1=b$.
    \smallskip

    To see that $\psi\circ\phi$ induces the identity on homotopy 
    classes, it is enough to note that, choosing a 
    $2$--simplex $c$ with $c_0=c_1=c_2=\sigma_0\partial_0b$,  
    $(b,\sigma_0\partial_0b)*(\sigma_0\partial_0b',b')$ is 
    homotopic to $(b,b')$ when $\partial_0b=\partial_0b'$. 
    This same observation suffices to show  that $\phi\circ\psi$ 
    induces the identity on homotopy classes since we may 
    suppose that $\hat p$ is in normal form. Hence $\tilde\Sigma_*(\cC)$ 
    and $\Sigma_*(\hat\cC)$ have isomorphic homotopy groupoids.
\end{proof}

   In view of this result we will denote the homotopy groupoid of 
   these symmetric simplicial sets by $\pi_1(\cC)$. Since $\pi_1(\cC)$ 
   and $\pi_1(\cC^{\circ})$ can both be calculated using $\Sigma_*(\hat\cC)$, 
   $\pi_1(\cC)$ and $\pi_1(\cC^{\circ})$ are isomorphic despite the 
   the fact that the categories can apparently have very little to do with 
   each other.

\subsection{Simplicial sets of a poset}
\label{Xb}

In the present paper we will work with the simplicial 
set $\tilde \Si_*(K)$ associated with a poset $K$ 
showing that it is equivalent to that used 
in \cite{RR06}. This will allow us to connect 
the cohomology of posets, developed 
in that paper, with bundles over posets.\bigskip

To begin with, we observe that a poset 
is a category with at most one arrow between any two objects. 
This implies that any covariant functor between 
the categories associated with two posets induces a unique 
order preserving map between the posets: the mapping between 
the corresponding set of objects. 
Conversely, any order preserving map between
two posets  induces a functor between 
the corresponding categories.\smallskip

Now, according to the previous observation, 
any $n$--simplex $x$ of the symmetric 
simplicial set $\tilde \Si_*(K)$, defined in the previous section, 
can be equivalently defined as an order preserving map $f:P_n\to K$. 
A detailed description of $\tilde \Si_*(K)$  in these terms 
has been given in \cite{RR06}. Here we just observe that  
a $0$--simplex is just an element of the poset. 
For $n\geq 1$, an $n-$simplex $x$ is formed
by $n+1$  $(n-1)-$simplices $\partial_0x, \ldots,\partial_nx$,
and by a $0$--simplex $|x|$ called the \emph{support}
of $x$ such that $|\partial_0x|, \ldots,|\partial_{n}x|\leq |x|$.
The nerve $\Si_*(K)$ turns out to be a subsimplicial set of $\tilde\Si_*(K)$. 
To see this, it is enough to define 
$f_0(a):= a$ on  $0$--simplices and, inductively,  
$|f_n(x)|:= \partial_{01\cdots (n-1)}x$ and 
$\partial_i f_n(x):= f_{n-1}(\partial_ix)$. 
So we obtain a simplicial map 
$f_*:\Si_*(K)\to \tilde \Si_*(K)$\footnote{In \cite{RR06} 
$\Si_*(K)$ was called the inflationary structure 
of  $\tilde \Si_*(K)$ and denoted by $\Si^{\inf}(K)$. 
This part  of $\tilde \Si_*(K)$ 
encodes the order relation of $K$, explaining the terminology.}.
We sometimes adopt the following 
notation: $(o;a,\tilde a)$ is the
1--simplex of $\tilde\Si_1(K)$ whose support is $o$ and 
whose 0-- and 1--face are, respectively, $a$ and $\tilde a$;
$(a,\tilde a)$ is the 1--simplex of the nerve $\Si_1(K)$
whose 0-- and 1--face are, respectively, 
$a$ and $\tilde a$.\smallskip

In the present paper we will consider pathwise connected posets $K$. 
This amounts to saying that the simplicial set $\tilde\Si_*(K)$ is pathwise 
connected. The homotopy groupoid $\pi_1(K)$  
of $K$ is defined as $\pi_1(\tilde\Si_*(K))$. Correspondingly, when 
$\pi_1(K)$ is trivial we will say that $K$ is \emph{simply connected}.
Now, a poset $K$ is \emph{upward} directed whenever 
for any pair $o,\hat o\in K$ there is $\tilde o$ such that 
$o,\hat o\leq \tilde o$. It is \emph{downward} directed if the dual poset 
$K^\circ$ is upward directed. When $K$ is upward directed, then 
$\tilde\Si_*(K)$ admits a contracting 
homotopy. So in this case $K$ is simply connected.  
However,  since $\pi_1(\tilde\Si_*(K))$ is isomorphic
to  $\pi_1(\tilde\Si_*(K^{\circ}))$, 
$K$ is simply connected whenever $K$ is downward directed too. 
The link between the first homotopy group of a poset and 
the corresponding topological notion 
can be achieved as follows. Let $M$ be an arcwise connected 
manifold, consider a basis for the topology of $M$ whose elements
are connected and simply connected, open  subsets of $M$.
Denote the poset formed by ordering this basis  
under \emph{inclusion} by $K$. 
Then $K$ is pathwise connected and $\pi_1(K)$ turns out to be
isomorphic to the fundamental groupoid $\pi_1(M)$ of $M$ 
\cite[Theorem 2.18]{Ruz05}.

   
\subsection{Posets as topological spaces} 
\label{Xc}    
   Let $K$ be a poset which we will equip with a topology defined by 
   taking $V_a:=\{a'\in K:a'\leq a\}$ as a base of neighbourhoods for 
   the topology. This topology corresponds to the Alexandroff topology 
   on $K^{\circ}$. The reason for choosing this convention is that the 
   map $a\mapsto V_a$ is an order isomorphism so that any poset 
   is a base for a topology ordered under inclusion. We denote $K$ 
   equipped with this topology by $\tau K$. It is easy to verify that a 
   mapping $f$ of posets is order preserving if and only if the corresponding
   mapping $\tau f$ of topological spaces is continuous. Thus the category 
   $\cK$ of posets may be regarded as a full subcategory of the category 
   of topological spaces.\smallskip 

  Clearly if we just know that $K$ is a base for a topology then this 
  cannot yield more information than can be got by supposing that 
  $\tau K$ is the topological space in question\footnote{
  Knowing that the space in question is Hausdorff could yield 
  more information.}.  Therefore it is 
  worth recalling the principal features of the topological space. 
 $\tau K$ is a $T_0$--space but not Hausdorff unless the ordering 
 is trivial. The components coincide with the arcwise connected 
components and the associated Hausdorff space is the space of 
 components with the discrete topology. The open set $V_a$ is 
the smallest open set containing $a$. It has a contracting homotopy 
and is hence arcwise connected and simply connected. It therefore 
follows from \cite{Rob03} that the path connected components of $K$ are in 
$1$--$1$ correspondence with the arcwise connected components 
of $\tau K$ and that $\pi_1(K)$ is isomorphic 
to $\pi_1(\tau K)$. Finally, we will call the open covering $\cV_0$ 
of $K$ defined by $\cV_0:=\{V_a: a\in K\}$ the 
\emph{fundamental covering}. Note that if $\cV$ is any other open 
covering of $K$, then any $V_a$ is contained in some element of $\cV$.
In the following we denote 
$V_{a_1}\cap V_{a_2}\cap \cdots\cap V_{a_n}$ by $V_{a_1a_2\cdots a_n}$.
A function $f$ from a poset $K$ to a set  $X$ 
is \emph{locally constant} whenever $f(o)=f(\tilde o)$ for 
$o\leq \tilde o$, or, equivalently, if $f(\partial_1b)=f(\partial_0b)$ 
for any 1--simplex $b$ of the nerve. Another equivalent description 
is the following: if $X$ is enodwed  with the trivial order, then 
$f:\tau K\to \tau X$ is continuous.

\section{Bundles over posets}
\label{A}

In this section we deal with bundles over posets without 
any particular structure. We discuss morphisms between bundles, 
connections and flat connections. We finally study the local
properties, like local triviality and the existence of local cross
sections. Throughout this paper the poset $K$ is assumed to be pathwise
connected. 
\subsection{Net bundles}
\label{Aa}
We introduce the notion of a net bundle over a poset. As we shall
see, a net bundle is a fibration of 
symmetric simplicial sets associated with  
posets (see Section \ref{Xb}),  equipped 
with a cocycle of the nerve of the poset base,  with values in the fibres.
This cocycle replaces continuity (or differentiability) 
in the theory of bundles over manifolds.
\begin{definition}
\label{Aa:1}
A \textbf{net bundle} over a poset $K$ is formed by 
a set $B$,  a surjective map $\pi:B\to K$, and 
a  collection $J$ of bijective mappings 
\[
J_{b}:\pi^{-1}(\partial_1b)\to
\pi^{-1}(\partial_0b), \qquad b\in \Si_1(K), 
\]
satisfying the following relations: 
\begin{itemize}
\item[(i)]   $J_{\partial_0c} \, J_{\partial_2c} = J_{\partial_1c}$  
             for any $c\in \Si_2(K)$,  
\item[(ii)]  $J_{\si_0a}  = \mathrm{id}_{\pi^{-1}(a)}$, 
             for any $a\in\Si_0(K)$. 
\end{itemize}
The symbol $\cB$ will denote the net bundle whose 
data is $(B, \pi,J,K)$.
\end{definition}
As usual $B$,  $K$ and  $\pi$ are called, respectively, the
\textit{total} 
space,  the \emph{base} space and the 
\emph{projection}. The subset $B_o :=\pi^{-1}(o)\subset B$
is called the \emph{fibre} over $o$. The  collection 
$J$ is called the \emph{net structure} of the bundle. 
Note, in fact that the correspondence 
$K\ni o \rightarrow B_o\subseteq B$ with $J$ is a net: given $o\leq \tilde o$ 
then $\mathrm{ad}_{J_{(\tilde o, o)}}:B_o\to B_{\tilde o}$, 
where $\mathrm{ad}$ denotes the adjoint action, and 
$(\tilde o, o)$ the  1--simplex of the nerve having 0--  and 1--face 
respectively $\tilde o$ and $o$. In the following symbol 
$\hat\cB$  will indicate 
the net bundle $(\hat B,\hat\pi, \hat J,K)$. \smallskip 

Since $K$ is pathwise connected and the net structure 
is a bijection, all the fibres of a net bundle are isomorphic. 
Note, however that we can weaken the above definition assuming the net
structure to be just injective. Clearly, in this case the fibres are, 
in general, not isomorphic. We refer to these  bundles as \emph{quasinet
bundles}. In the present paper we will not deal with quasinet
bundles. However they will play a role in  K-theory \cite{RRV2}. \smallskip

The total space $B$ of a net bundle $\cB$ seems 
to have no structure. But this is not 
the case since there is an associated 
fibration of simplicial sets. First of all, note
that given $\psi,\phi\in B$ and writing 
\[
 \psi\leq_{J} \phi \ \ \iff \ \ \pi(\psi)\leq \pi(\phi) \ \mbox{ and } \ 
 J_{\pi(\phi),\pi(\psi)}\psi = \phi 
\]
yields  an order relation on $B$. So, the total space is a poset too 
with the fibres $B_o$ having the discrete ordering. This, in turns, 
implies that  $\cB$ is a fibration of simplicial sets. Consider  
the symmetric simplicial set $\tilde\Si_*(B)$. Set 
$\pi_0(\phi) :=\pi(\phi)$ for any $\phi\in\tilde\Si_0(B)$, and
by induction,
for $n\geq 1$,  define 
\[
  \pi_n(x)  := (\pi_{0}(|x|), \pi_{n-1}(\partial_0x), \ldots , 
  \pi_{n-1}(\partial_nx)), 
\qquad x\in\tilde\Si_n(B).
\]
Then, we have  a surjective simplicial map 
$\pi_*:\tilde\Si_*(B)\to \tilde\Si_*(K)$.  
This map is symmetric, i.e., $\pi_n\tau_i = \tau_i\pi_n$, and
preserves the nerves, i.e., $\pi_*:\Si_*(B)\to\Si_*(K)$. 
Now, the \emph{fibre} over  a $0$--simplex $a$ is the simplicial set 
$\pi^{-1}_*(a)$ defined by $\pi^{-1}_0 (a)$ and 
$\pi^{-1}_{k+1} (\si_{k\cdots 10}a)$ for $k\geq 0$.
So we have a fibration of symmetric simplicial sets whose 0--fibres are,
according to our definition, isomorphic. We observe that this
is not a Kan fibration since the simplicial sets involved, in general, 
do not fulfill
the extension condition \cite{RR06}.\smallskip

We now provide a first example  of a
net bundle, the \emph{product net bundle}, 
introducing some notation useful in what 
follows. Nontrivial examples, 
will be given in Section \ref{Da}.
Let $K$ be a poset, and let $X$ be a space. 
Consider  the Cartesian product  $K\times X$ and  define 
$\pi(o, x) := o$.  
Clearly, $\pi: K\times X\rightarrow K$ is a surjective map, and 
$\pi^{-1}(o)\simeq o\times X$. Let 
\[
\jmath_{b}(\partial_1b,x) := (\partial_0b,x), \qquad (b,x)\in 
\Si_1(K)\times X. 
\]
Clearly, $\jmath_{b}:\pi^{-1}(\partial_1b)\to\pi^{-1}(\partial_0b)$
is a bijective map.\smallskip

We now introduce morphisms between net bundles. 
Since in the present paper we will not need to compare net bundles 
over different posets, we will consider only   morphisms 
leaving the base space invariant. The general definition 
can be easily obtained mimicking that
for bundles over manifolds. 
\begin{definition}
\label{Aa:4}
Let $\cB$ and 
$\hat\cB$ be  net bundles 
over the poset $K$.  A \textbf{net bundle morphism} $f$ from 
$\cB$  into  $\hat\cB$ is a mapping of total spaces $f:B\to \hat B$ 
preserving the fibres  and commuting with the net structure,
namely 
\begin{itemize} 
\item[(i)] $\hat \pi\, f = \pi$; 
\item[(ii)] $\hat J\, f = f\, J$.
\end{itemize}
The bundles $\cB$ and $\hat \cB$ are said to be \textbf{isomorphic} whenever 
$f:B\rightarrow \hat B$ is a bijection. A bundle isomorphic to 
a product bundle is said to be \textbf{trivial}.
\end{definition}
We conclude this section by defining the restriction of a net bundle 
used later to define  local triviality. 
Let $\cB$ be a net bundle over $K$. Given
an open and pathwise connected subset $U$ of $K$, define 
$  B|_{U}   :=   \pi^{-1}(U)$, and 
$\pi|_U(\phi)  :=  \pi(\phi)$ for $\phi\in \cB|_{U}$.
Moreover   set ${J|_U}_{b} := J_{b}$ for any 
$b\in \Si_1(U)$. Then, $\pi_U:B|_{U}\to U$ is a surjective map
and  $J|_U$ a net structure. We call this bundle 
the \emph{restriction} of $\cB$ to $U$ and denote it by 
$\cB|_U$. 
\subsection{Connections}
\label{Ab}
We introduce the notion of a connection 
on a  net bundle, and related notions like 
parallel transport and flat connections.\bigskip 

Since the net structure $J$ of a net bundle $\cB$ over $K$ 
is a bijection between the fibres it admits an extension 
from the nerve $\Si_1(K)$ to the simplicial set $\tilde\Si_1(K)$.
In fact, let 
\begin{equation}
\label{Ab:1}
Z(b) := J^{-1}_{(|b|, \partial_0b)}\,
           J_{(|b|,\partial_1b)}, \qquad b\in\tilde\Sigma_1(K),
\end{equation}
where $(|b|,\partial_ib)$ is the 1--simplex of the nerve with 
$0$--face $|b|$ and 1--face 
$\partial_ib$, for $i=0,1$. This is well posed since
the support of a 1--simplex of $\tilde\Si_1(K)$
is greater than its faces (see Section \ref{Xb}). 
So we have a bijection $Z(b): B_{\partial_1b}\to B_{\partial_0b}$ 
for any $b\in\tilde\Si_1(K)$. Using the defining properties 
of the net structure we have  
\begin{align*}
Z(\partial_0c)\, Z(\partial_2c) &
   = J^{-1}_{(|\partial_0c|, \partial_{00}c)}\, 
           J_{(|\partial_0c|,\partial_{10}c)}\,
    J^{-1}_{(|\partial_2c|, \partial_{02}c)}\,
           J_{(|\partial_2c|,\partial_{12}c)}  \\
 & =  J^{-1}_{(|c|, \partial_{00}c)}\,
           J_{(|c|,\partial_{10}c)}\,
    J^{-1}_{(|c|, \partial_{02}c)}\,
           J_{(|c|,\partial_{12}c)}  \\
 & =  J^{-1}_{(|c|, \partial_{01}c)}\,
           J_{(|c|,\partial_{10}c)}\,
    J^{-1}_{(|c|, \partial_{10}c)}\,
           J_{(|c|,\partial_{11}c)}  \\
 & =  J^{-1}_{(|c|, \partial_{01}c)}\,
           J_{(|c|,\partial_{11}c)}  \\
 & =  Z(\partial_1c),
\end{align*}
for any $c\in\tilde\Si_2(K)$,
where the commutation relations of the faces,
$\partial_{ij}= \partial_{j-1,i}$, if 
$i<j$, have been used (recall that $\partial_{ij}$ stands for 
$\partial_i\partial_j$). \smallskip

After this observation we are in a position to define connections. 
\begin{definition}
\label{Ab:2}
A \textbf{connection} on a net bundle $\cB$
is a field $U$ associating 
a bijective mapping  $U(b):B_{\partial_1b}\to B_{\partial_0b}$ 
to any $1$--simplex $b$ of $\tilde\Si_1(K)$, and  
such that  
\begin{itemize}
\item[(i)]  $U(\overline{b})   =   U(b)^{-1}$,  for  $b\in\tilde\Si_1(K)$; 
\item[(ii)]  $U(b)  =  J_{\partial_1b,\partial_0b}$ for any 
        $b\in \Si_1(K)$.
\end{itemize}
A connection $U$ is said to be \textbf{flat} whenever
\[ 
U(\partial_0c) U(\partial_2c) =  U(\partial_1c), \qquad
c\in\tilde\Si_2(K).
\]
We denote the set of connections of a net bundle $\cB$ 
by $\sU(K,\cB)$.
\end{definition}
According to this definition a connection is any extension 
of the net structure to the simplicial set $\tilde \Si_1(K)$; the extension 
$Z$ defined by (\ref{Ab:4}) is a flat connection 
of the net bundle. This connection is characteristic of 
the net bundle, as the following proposition shows.
\begin{proposition}
\label{Ab:3}
$Z$ is the unique  flat connection of $\cB$.
\end{proposition}
\begin{proof}
Given a 1--simplex $b\in\tilde\Si_1(K)$ consider the 2--simplex 
$c\in\tilde\Si_2(K)$ defined by 
\[
 |c|=|b|, \ \ \partial_2c=(|b|;|b|,\partial_1b), \ \  
              \partial_0c=(|b|;\partial_0b, |b|), \ \
              \partial_1c=b.
\]
Observe that $\partial_2c$ and the reverse of $\partial_0c$ are
1--simplices of the nerve. If $U$ is a flat connection, then
$U(b) = U(\partial_0c)\, U(\partial_2c)$
$= J^{-1}_{(|b|,\partial_0b)}\, J_{(|b|,\partial_1b)}$ $ = Z(b)$.
\end{proof}
This result is one of the main differences from 
the theory of bundles over manifolds. As 
a direct consequence, we shall see in Section \ref{Da} that  
principal net bundles over a simply connected poset 
are trivial. Concerning non-flat connections, we point out 
that, except in trivial situations, the set of non-flat connections 
of a principal net bundle is never 
empty (see Section \ref{Da}).\smallskip   

Finally given  a path  $p$ 
of the form $p=b_n*\cdots *b_1$, 
the \emph{parallel transport} along $p$, induced by a
connection $U$, is  the mapping 
$U(p):B_{\partial_1p}\rightarrow B_{\partial_0p}$ 
defined by 
\begin{equation}
\label{Ab:4}
 U(p)  := U(b_n)\, \cdots \, U(b_2) \, U(b_1) \ .
\end{equation}
The defining properties of a connection  imply 
$U(\overline{p})=U(p)^{-1}$ for any 
path $p$, and $U(\si_0a)=\mathrm{id}_{B_a}$ for any 0--simplex $a$.


\subsection{Local triviality}
\label{Ac}

We aim to show that any net bundle $\cB$ 
is \emph{locally trivial}. Thus there is
a set $X$ and an open covering $\cV$ of the poset $K$ such that,
for any $X\in\cV$, the restriction  
$\cB|_V$ is equivalent to the product net bundle 
$V\times X$. In particular we will show that this holds 
for the fundamental covering  $\cV_0$.
This suffices for our purposes. In fact if a net bundle can be 
trivialized on a covering $\cV$, then it can be 
trivialized on the fundamental covering $\cV_0$, too
because $\cV_0\subseteq \cV$.
\begin{proposition}
\label{Ac:1}
Any net bundle $\cB$ over $K$ can be trivialized 
on the fundamental covering $\cV_0$ of $K$: 
thus there is a set $X$ such that the restriction 
$\cB|_{V_a}$ is isomorphic to the product bundle $V_a\times
X$, for any $a\in K$.
\end{proposition}
\begin{proof}
Fix $o\in K$ and define 
$X := \pi^{-1}(o)$. Using the flat connection 
$Z$ of $\cB$ and pathwise connectedness, we can find a bijection 
$F_{a}: X\rightarrow \pi^{-1}(a)$ for any $0$--simplex $a$
(clearly  $F_{a}$ is not uniquely determined). 
Now recalling the definitions of restriction and product bundle 
we must show that there is a family of bijective mappings 
$\theta_a: V_a\times X\to \pi^{-1}(V_a)$, with $a\in K$,  
such that $\pi\,\theta_a=pr_1$ and 
$J_b\,\theta_a = \theta_a\,j_b$ for any $b\in\Si_1(V_a)$. 
To this end define 
\[
 \theta_a(o,v) :=  J_{(o,a)} F_a(v), \qquad (o,v)\in V_a\times X.
\]
Clearly $\theta_a: V_a\times X\to \pi^{-1}(V_a)$, 
being the composition of bijective maps, is bijective, 
and 
$\pi \theta_a = pr_1$. Moreover, 
If $b\in \Si_1(V_a)$, then 
\[
 J_{b}\, \theta_a (\partial_1b,v)  =  
 J_{b}\,  J_{(\partial_1b,a)}\,  F_a (v)
 =  J_{(\partial_0b,a)}\,  F_a (v) =  \theta_a\,  \jmath_{b} (\partial_1b,v),
\]
and this completes the proof. 
\end{proof}
The set  $X$ is called the \textit{standard fibre}
of $\cB$. A family of mappings $\theta:=\{\theta_a\}$, with $a\in K$, 
trivializing the net bundle $\cB$ on the fundamental covering 
will be called a \emph{local trivialization} of $\cB$.\smallskip

We now deal with local sections of net bundles.  
\begin{definition}
\label{Ac:2}
A \textbf{local section} of a net bundle $\cB$  is a map 
$\si: V\rightarrow B$, where $V$ is an open of $K$, 
such that $\pi\,\si =\mathrm{id}_V$, and 
\[
 J_{b}(\si(\partial_1b)) = \si(\partial_0b), \qquad b\in \Si_1(V).
\]
If $V=K$, then $\si$ is said to be a \textbf{global} section.
\end{definition}
Consider a local section $\si:V\rightarrow B$. Let  $a$ be 0--simplex 
with $a\in V$, and recall that $V_a\subseteq V$. Given a local
trivialization $\theta$, define 
\begin{equation}
\label{Ac:3}
 \theta^{-1}_a(\si(o))  := (o,s_a(o)),  \qquad o\in V_a.
\end{equation}
We call $s_a$ a \emph{local representative} 
of $\si$. 
An important property of net bundles,  is that 
cross sections are \emph{locally constant} (see Subsection \ref{Xc}). 
In fact,  $(o, s_a(o))  = \theta^{-1}_a \si(o)$ 
$= \theta^{-1}_a\,  J_{(o,a)}\si(a)=$ 
$ \jmath_{o,a}\,  \theta^{-1}_a \si(a)$ $= (o, s_a(a))$, for any 
$o\in V_a$.
A second property is that any net bundle has local cross 
sections. In fact, given $a\in\tilde\Sigma_0(K)$, pick 
$\phi\in\pi^{-1}(a)$ and define 
$\si_a(o) := J_{(o,a)}(\phi)$ for $o\in V_a$. 
One can easily sees that $\si_a:V_a\to B$ is a local cross section. 
%
%
\section{Principal bundles over posets} 
\label{B}

We now introduce the notion of a principal net bundle over a poset. 
This is a net bundle with a suitable action of a group. 
Many of the previous notions, like
morphisms and connections,  generalize straightforwardly by  
requiring equivariance.  We study local trivializations, 
introduce transition functions and  point out their main 
feature: they are locally constant.\bigskip

\begin{definition}
\label{B:1}
A \textbf{principal net bundle}  $\cP$ over a poset $K$  
is a net bundle $(P,\pi,J,K)$ with a  group $G$ acting 
freely on the total  space $P$ on the right.   
The action $R$  preserves the fibres,  
$\pi\,  R_g=\pi$ for any $g\in G$, 
is transitive on the fibres, and is compatible with  
the net structure, namely  $J \,  R = R\, J$.
We call $G$  the \textbf{structure group} of $\cP$.  
We denote  the set of  principal net 
bundles having structure group $G$ by $\sP(K,G)$.
\end{definition}
In the sequel we adopt the following notation for the action 
of the structure group: $R_g(\psi):= \psi\cdot g$, with 
$\psi\in P$ and $g\in G$.\smallskip

It is clear from this definition that the fibres of a principal net bundle
are all isomorphic to the structure group. Furthermore, the relevant 
topology of the structure  is the discrete one, since 
the order induced by the net structure on the total space is 
trivial when restricted to the fibres (see Section \ref{Aa}).
Hence the topology induced on the fibres  is discrete. \smallskip

As an example, 
given a group $G$, consider the product net bundle $K\times G$ 
introduced in the previous section. Define 
\begin{equation}
\label{B:1a}
  r_h(o,g) := (o,gh), \qquad  o\in K, \  g,h\in G.
\end{equation}
Clearly we have a principal net bundle that we call the 
\emph{product principal net bundle}.\smallskip

We now introduce principal morphisms and the category associated 
with principal net bundles.
\begin{definition}
\label{B:2}
Consider two principal net  bundles 
$\cP, \hat \cP\in\sP(K,G)$. A \textbf{morphism} 
$f$ from $\hat\cP$ into $\cP$ 
is an equivariant net bundle morphism $f:\hat P\to P$, namely  
$R \,  f   =   f\,  \hat R$. 
We denote the set of morphisms from 
$\hat \cP$  to $\cP$  by $(\hat\cP,\cP)$.
\end{definition}
The definition of morphisms involves only 
principal net bundles over the same poset $K$ and having the same
structure group. In Section \ref{Db}, we will see how to connect 
principal net bundles having different structure groups. Now, 
given $\cP,\hat\cP, \tilde\cP\in P(K,G)$, let 
$f_1\in (\cP,\hat\cP)$ and $f_2\in (\hat\cP,\tilde\cP)$. Define 
\begin{equation}
\label{B:3}
 (f_2\, f_1) (\phi)  := f_2(f_1(\phi)),\qquad \phi\in P.
\end{equation}
It is easily seen that $f_2\, f_1\in (\cP,\tilde\cP)$. 
The composition law (\ref{B:8}) makes $\sP(K,G)$ into a category,
the category of \emph{principal net bundles with structure group
 $G$}. We denote this
by the same symbol as used to denote the set of objects.
The identity $1_{\cP}$ of $(\cP,\cP)$ is the identity 
automorphism of $\cP$. Since we are considering morphisms between 
principal bundles with
the same structure group, it is easily seen,  that any  morphism 
$f\in (\hat\cP,\cP)$ is indeed an \emph{isomorphism}, namely  
there exists a morphism $f^{-1}\in (\hat\cP,\cP)$
such that $ f\,  f^{-1} = 1_{\cP}$ and 
$f^{-1}\,  f = 1_{\hat\cP}$.
On these grounds,  given 
$\hat \cP, \cP\in\sP(K,G)$, and writing 
\begin{equation}
\label{B:4}
\hat\cP\cong\cP \ \iff \ (\hat\cP,\cP)\ne\emptyset.
\end{equation}
we endow $\sP(K,G)$ of an equivalence relation $\cong$ and we shall say 
that $\cP_1$ and $\cP$ \emph{are equivalent}. A principal 
net bundle $\cP\in\sP(K,G)$ is said to be \emph{trivial} 
if it is equivalent to the principal product  net bundle $K\times G$.\bigskip

Following the scheme of the previous section, 
we now deal with  the notion of a connection 
on a principal net bundle, and related notions like 
parallel transport, holonomy group, 
and flat connection. In particular we analyze this property 
from a global point of view, i.e. without using local trivializations. 
\begin{definition}
\label{B:5}
A \textbf{connection} on a principal net bundle $\cP$
is a net bundle connection $U$ of $\cP$ which is 
equivariant, namely 
\[
U(b) \, R  =   R \, U(b), \qquad  b\in\tilde\Si_1(K).
\]
We denote the set of connections of a principal net bundle $\cP$ 
by $\sU(K,\cP)$.
\end{definition}
The notions of  parallel transport along paths and 
flatness  for generic net bundles 
admit  a straightforward generalization to principal 
net bundles. The only important point 
to note is that if $U$ is a connection of $\cP$, 
and $p$ is a path, then the parallel 
transport $U(p)$ is an equivariant map 
from $P_{\partial_1p}$ to $P_{\partial_0p}$.
The \emph{holonomy and the restricted 
holonomy} group of the connection $U$  with respect to the base point
$\psi\in P$ are defined by  
\begin{equation}
\label{B:6}
\begin{array}{c}
H_U(\psi) := \{ g\in G \ | \ \psi\cdot g =  U(p)\psi, \ \
\partial_0p=\pi(\psi)=\partial_1p\} ; \\[3pt]
H^0_U(\psi) := \{ g\in G \ | \ \psi\cdot g =  U(p)\psi, \ \
\partial_0p=\pi(\psi)=\partial_1p, \ \ p\sim \si_0a\} . 
\end{array}
\end{equation}
One  sees that both $H_U(\psi)$ and $H^0_U(\psi)$
are indeed subgroups of $G$ and that $H^0_U(\psi)$ is a normal subgroup 
of $H_U(\psi)$.  This  will be shown 
in Section \ref{Db} where we will deal with the cohomology of 
principal of net bundles and relate holonomy  with  reduction theory. 
We finally observe that a principal net bundle has a unique 
flat connection $Z$ defined by 
equation (\ref{Ab:1}).\bigskip

Let  $U,\hat U$ be a pair of connections of the principal net bundles 
$\cP,\hat\cP\in \sP(K,G)$ respectively. The set $(\hat U,U)$ of 
the \textbf{morphisms} from $\hat U$ to $U$ is the subset 
of the morphisms $f\in(\hat \cP,\cP)$ such that 
$f\, \hat U = U\, f$. It is clearly an equivalence relation
as any element of $(\hat\cP,\cP)$ is invertible. 
It is worth observing that, given a principal net bundle 
$\cP$, then $(Z,Z) = (\cP,\cP)$  
($Z$ is the flat connection on $\cP$). \emph{The category of connections of
principal net bundles over $K$ with structure group $G$,} 
is the category whose objects are connections on principal net bundles 
of $\sP(K,G)$ and whose set of arrows are the corresponding
morphisms. We denote this category by $\sU(K,G)$. 
Furthermore, given $\cP\in\sP(K,G)$ 
we call the \emph{category of connections on $\cP$}, the full
subcategory of $\sP(K,G)$ whose set of objects is $\sU(K,\cP)$. 
We denote this category by $\sU(K,\cP)$ 
as for  the corresponding set of objects. 
\begin{lemma}
\label{B:7}
The category $\sP(K,G)$ is isomorphic 
to the full subcategory $\mathscr{U}_f(K,G)$ 
of  $\mathscr{U}(K,G)$ whose objects are flat connections.
\end{lemma}
\begin{proof}
Given $\cP\in\sP(K,G)$, define 
$F(\cP) := Z$ where $Z$ is the unique flat connection of $\cP$.
For any $f\in(\cP,\cP_1)$ define $F(f)=f$. Clearly 
$F:\sP(K,G)\rightarrow \mathscr{U}_f(K,G)$ is 
a covariant functor. It is full because the set of morphisms 
$(\cP,\cP_1)$ equals the set of the morphisms of $(Z,Z_1)$. 
It is an isomorphism because for any $Z$, there obviously exists 
a principal net bundle $\cP$, the bundle where $Z$ is defined, 
such that $F(\cP)=Z$. 
\end{proof}

We have seen that any  net bundle is locally trivial 
(Proposition \ref{Ac:1}). 
We now   show that principal net 
bundles are locally trivial too. This will allow us to 
investigate   the local behaviour of  
trivialization maps, cross sections and morphisms. In particular we will
point out the main feature of these local 
notions: all of them are, in a suitable sense, locally constant.
\begin{proposition}
\label{B:8}
Any principal net bundle $\cP$ admits 
a local trivialization on the fundamental covering $\cV_0$.
\end{proposition}
\begin{proof}
The proof is slightly different from the proof of Proposition \ref{Ac:1},
because equivariance is required. 
For any  $0$--simplex $a$,  choose an element 
$\phi_a\in P_a$, and define 
\begin{equation}
\label{B:9}
 \theta_a(o,g)  :=  J_{(o,a)}(\phi_a)\cdot g,  
  \qquad (o,g)\in V_a\times G.
\end{equation}
As the action of the group is free $\theta_a:V_a\times G\rightarrow 
\pi^{-1}(V_a)$ is injective. 
Moreover, given  $\phi\in \pi^{-1}(V_a)$ observe that 
$J_{(\pi(\phi),a)}(\phi_a)$ and $\phi$ belong to the same fibre.  
Since the action of $G$ is transitive on the fibres, 
there exists $g_a(\phi)\in G$
such that $\phi= J_{(\pi(\phi),a)}(\phi_a)\cdot g_a(\phi)$. 
Then $\theta_a(\pi(\phi),g_a(\phi))= 
J_{(\pi(\phi),a)}(\phi_a)\cdot g_a(\phi)=\phi$. This proves that 
$\theta_a$ is bijective. Now, given $b\in \Si_1(V_a)$, we have  
\[
 J_{b}\, \theta_a\, (\partial_1b,g)  = 
  J_{b}\,J_{\partial_1b,a}(\phi_a) \cdot g 
   =  J_{\partial_0b,a}(\phi_a)\cdot g 
    =  \theta_a(\partial_0b,g)
   = \theta_a\, \jmath_{b}\,(\partial_1b,g).
\]
Finally, it is clear that 
$R\,\theta_a = \theta_a\, r$. 
\end{proof}
As for  net bundles,   a \emph{local trivialization}
of a principal net bundle $\cP$ means  a family of mappings 
$\theta:= \{\theta_a\}$, $a\in K$, trivializing $\cP$ on the 
fundamental covering.\smallskip

Consider a local trivialization $\theta$ of a principal net bundle
$\cP$. Given  a pair of 0--simplices $a,\tilde a$, define 
\begin{equation}
\label{B:10}
\theta^{-1}_{\tilde a} \, \theta_{a} (o, g)  :=  
(o, z_{\tilde aa}(o)g), \qquad o\in V_{\tilde aa}
\end{equation}
where $z_{\tilde aa}(o)\in G$. This definition is well posed 
because of the equivariance of 
$\theta_a$. So we have a family $\{z_{\tilde aa}\}$  
of functions $z_{\tilde aa}:V_{\tilde aa}\rightarrow G$ 
associated with the local trivialization $\theta$. 
\begin{lemma}
\label{B:11}
Under the above assumptions and notation, 
$z_{a\tilde a}:V_a\to G$  are locally constant maps 
satisfying the cocycle identity 
\[
 z_{\hat a\tilde a}(o)\,  z_{\tilde aa}(o)=   z_{\hat aa}(o), \qquad
 o\in V_{a\tilde a\hat a}.
\]
\end{lemma} 
\begin{proof}
Consider the 1--simplex $(o_1,o)$ of the nerve. By the defining 
properties  of a local trivialization (see proof of  
Proposition \ref{B:8})  we have $J_{(o_1,o)}\, \theta_a =
\jmath_{(o_1,o)}\, \theta_a$. This  implies that 
$\theta_a\, \jmath_{(o_1,o)}\,  \theta^{-1}_a
=\theta_{\tilde a}\,  \jmath_{(o_1,o)} \,\theta^{-1}_{\tilde a}$. 
Hence 
\[
 (o_1, z_{\tilde aa}(o_1))  
 =   \theta^{-1}_{\tilde a}\,   \theta_{a}\,  
    \jmath_{(o_1,o)} \,(o,e)
 = \jmath_{(o_1,o)}\,  \theta^{-1}_{\tilde a}\,
  \theta_{a} \, (o,e) 
 = (o_1,z_{\tilde aa}(o)),
\]
and this proves that these maps  are locally constant. The cocycle
identity is obvious.
\end{proof}
The functions $\{z_{\tilde aa}\}$  defined by equation (\ref{B:10})
will be called \emph{transition functions} of the principal net
bundles. The cocycle identity says that transition functions 
can be interpreted in terms of the \v Cech cohomology 
of posets. This aspect will be developed in Section \ref{E}. 
Finally, the dependence of  transition functions 
on the local trivialization will be discussed in 
Section \ref{Ca} in terms of the associated 1--cocycles.\smallskip

Consider a cross section $\si:V\rightarrow P$.  
We have already seen that $\si$ is locally constant (see Subsection
\ref{Ac}). Now, 
it is easily seen that, if $o\in V_{\tilde aa}$ and $\tilde a,a\in U$, 
then  $s_{\tilde a}(o) =  z_{\tilde aa}(o)\, s_{a}(a)$.
\begin{lemma}
\label{B:13}
$\cP$ is trivial if, and only if, it admits a global section. 
\end{lemma}
\begin{proof}
$(\Rightarrow)$ Given  $f\in (K\times G,  \cP)$,
then $\si(o) := f(o,e)$ is a global cross section of
$\cP$.  $(\Leftarrow)$ Given a global section $\si$, define 
$f(o,g) := \si(o)\cdot g$ for any pair 
$(o,g)$. It is easily seen $f: K\times G\to P$ is an equivariant and 
fibre preserving bijective map. Furthermore, 
for any 1--simplex $b$ of the nerve we have 
\begin{align*}
   f\, \jmath_{b} (\partial_1b,g) & = f(\partial_0b,g) = 
    \si(\partial_0b)\cdot g
   =  J_{b}(\si(\partial_1b)) \cdot g \\
  & = J_{b}(\si(\partial_1b)\cdot g) 
   =  J_{b}\,  f (\partial_1b, g), 
\end{align*}
completing the proof.
\end{proof}
Morphisms between principal net bundles, 
like  sections and local trivializations, 
are locally constant. Consider $f\in (\hat\cP, \cP)$, and let 
$\hat\theta$ and $\theta$ be local trivializations of $\hat\cP$ and 
$\cP$ respectively. Define 
\begin{equation}
\label{B:12}
 (o,  f_a(o,g)) := 
   \theta^{-1}_a\ f\ {\hat\theta}_a  (o, g), \qquad 
   (o,g)\in V_a\times G.
\end{equation}
This equation defines, for any 0--simplex $a$, a function 
$f_a:V_a\times G \rightarrow G$ 
enjoying the following properties:
\begin{lemma} 
\label{B:14}
Under the above notation and assumptions, the function 
$f_a:V_a\times G \rightarrow G$ is  locally constant and 
$f_a(o,g) = f_a(o,e)g$, for any 
$(o,g)\in V_a\times G$.
\end{lemma}
\begin{proof}
Given $b\in \Si_1(V_a)$. Since $f$  $J\, f = f\, \hat J$,  we have 
\begin{align*}
 (\partial_0b,  f_a (\partial_0b,e)) & = 
   \theta^{-1}_a\, f\,  {\hat\theta}_a\,  (\partial_0b, e) 
 = \theta^{-1}_a\,  f\,
    {\hat\theta}_a\,  \jmath_{b}\,  (\partial_1b, e) \\
& =  \theta^{-1}_a\, f\, 
           {\hat J}_{b}\, {\hat\theta}_a\,(\partial_1b, e) 
 =  \theta^{-1}_a\,  J_{b} \, f\,
                  {\hat\theta}_a \, (\partial_1b, e) \\
& = \jmath_{b}\, \theta^{-1}_a\, f\,
                  {\hat\theta}_a \, (\partial_1b, e) = 
   \jmath_{b}\, (\partial_1b, f_a(\partial_1b,e)) 
\\
& =  (\partial_0b,  f_a(\partial_1b,e)).
\end{align*}
Hence $f_a$ is locally constant. Equivariance of $f$ completes 
 the proof.
\end{proof}

\section{Cohomological representation and equivalence}
\label{C}
Because of the lack of a differential structure, we replace 
the differential calculus of forms with a cohomology taking values 
in the structure groups of principal net bundles. 
In the present section we show that the 
theory of principal bundles and connections over
posets described in the previous section, 
admits an equivalent description in terms of a non-Abelian cohomology 
of posets \cite{RR06}. 
We will construct mappings 
associating  to geometrical objects like 
principal net bundles and connections, the corresponding 
cohomological objects: 1--cocycles and connections 1--cochains. 
We refer the reader to the Appendix for notation and a 
brief description of the results of the cited paper.

\subsection{Cohomological representation}
\label{Ca}
Consider a connection $U$ on a principal net bundle $\cP$. 
Let $\theta$ be a local trivialization 
of $\cP$. Given a 1--simplex $b$ of $\tilde\Si_1(K)$, define 
\begin{equation}
\label{Ca:1}
(\partial_0b,  \Gamma_\theta(U)(b) \cdot g)  := 
\theta^{-1}_{\partial_0b}\,  U(b) \, 
\theta_{\partial_1b}\, (\partial_1b,g),
\end{equation}
for any $g\in G$. By equivariance,    
this definition is well posed. This equation associates 
a 1--cochain $\Gamma_{\theta}(U):\tilde\Si_1(K)\rightarrow G$ to the 
connection $U$. 
\begin{proposition}
\label{Ca:2}
Given $\cP\in\sP(K,G)$,    let $\theta$
be a local trivialization of $\cP$. Then 
the following assertions hold: 
\begin{itemize}
\item[(i)] $\Gamma_\theta(U)\in \rU^1(K,G)$ for any $U\in \sU(K,\cP)$;
\item[(ii)] $\Gamma_\theta(Z)\in \rZ^1(K,G)$;
\item[(iii)] $\Gamma_\theta(U)\in 
 \rU^1(K,\Gamma_\theta(Z))$ for any $U\in \sU(K,\cP)$. 
\end{itemize}
\end{proposition}
\begin{proof}
It is convenient to introduce a new notation.
Given two 0--simplices 
$a$ and $a_1$ and an element $g\in G$, 
let $\ell_{a,a_1}(g): {a_1}\times G\to  a\times G$
be defined  by 
$\ell_{a,a_1}(g) (a_1,h) := (a,gh)$ for any  $h\in G$.
Then observe that equation (\ref{Ca:1}) can be rewritten as 
\begin{equation}
\label{Ca:3}
 \ell_{\partial_0b,\partial_1b} \big(\Gamma_{\theta}(U)(b)\big) = 
 \theta^{-1}_{\partial_0b}\,  U(b) \, \theta_{\partial_1b} \ .
\end{equation}
Now, given $U\in \sU(K,\cP)$ and a 1--simplex $b$, observe that 
\begin{align*}
\ell_{\partial_1b,\partial_0b}\big(\Gamma_{\theta}(U)(\overline{b})\big)
&  =  
\theta^{-1}_{\partial_1b}\, U(b)^{-1} \, 
\theta_{\partial_0b} 
 \\
& = 
\big(\theta^{-1}_{\partial_0b}\,  U(b)\, 
 \theta_{\partial_1b}\big)^{-1} = \ell_{\partial_1b,\partial_0b}\big(
 \Gamma_{\theta}(U)(b)^{-1}\big) \ .
\end{align*}
Hence $\Gamma_\theta(U)(\overline{b})= \Gamma_\theta(U)(b)^{-1}$. If 
$b$ is a 1--simplex of the nerve, then 
\[
\ell_{\partial_0b,\partial_1b}\big(\Gamma_{\theta}(U)(b)\big)  = 
 \theta^{-1}_{\partial_0b}\, J_{\partial_0b,\partial_1b}\,
\theta_{\partial_1b}
 = \theta^{-1}_{\partial_0b}\, Z(b)\, 
 \theta_{\partial_1b} 
  = \ell_{\partial_0b,\partial_1b}\big(\Gamma_{\theta}(Z)(b)\big) \ . 
\]
Hence $\Gamma_\theta(U)$  is equal to $\Gamma_\theta(Z)$
on 1--simplices of the nerve. It remains to show that 
$\Gamma_\theta(Z)$ is a 1--cocycle. Given a 2--simplex $c$ we have
\begin{align*}
\ell_{\partial_{00}c,\partial_{12}c} 
\big(\Gamma_{\theta}(Z)(\partial_0c)\, \Gamma_{\theta}(Z) & 
(\partial_2c)\big)  
   =  \\
& = \theta^{-1}_{\partial_{00}c}\,  Z(\partial_0c)\, 
     \theta_{\partial_{10}c}\,
\theta^{-1}_{\partial_{02}c}\,  Z(\partial_2c)\,
 \theta_{\partial_{12}c}\\
&  
 =  \theta^{-1}_{\partial_{01}c}\,
Z(\partial_1c)\,
 \theta_{\partial_{11}c}
=  \theta^{-1}_{\partial_{00}c}\,
Z(\partial_1c)\,
 \theta_{\partial_{12}c}\\
&  = \ell_{\partial_{00}c,\partial_{12}c}
\big(\Gamma_{\theta}(Z)(\partial_1c)\big) \ ,
\end{align*}
and this completes the proof.
\end{proof} 
Sometimes,  when no confusion is possible, 
we will adopt the following notation
\begin{equation}
\label{Ca:4}
\begin{array}{ll}
 u_\theta  := \Gamma_\theta(U), &  U\in\sU(\cP),\\[3pt]
 z_\theta  := \Gamma_\theta(Z), &  Z\in\sU(\cP),
\end{array}
\end{equation}
where $Z$ is the flat connection of $\cP$.
Since $Z$ is unique, we will refer to $z_\theta$
as the \emph{bundle cocycle} of $\cP$. 
Furthermore, 
we will call $u_\theta$ the \emph{connection cochain} of 
the connection $U$. We shall see in Proposition \ref{Ca:7} 
how these objects depend on the choice of  
local trivialization. Notice that 
the assertion $(iii)$ of the above proposition, 
says that for any connection $U\in\sU(K,\cP)$, 
$z_\theta$ is the 1--cocycle induced by $u_\theta$ 
(see Appendix).\smallskip

The next result shows the relation  between the bundle cocycle 
and transition functions of the bundle.
\begin{lemma}
\label{Ca:5}
Given  a principal net bundle  $\cP$,  let $\theta$ 
be a local trivialization and $\{ z_{a\tilde a}\}$   
the corresponding family of transition functions. Then 
\[
 z_\theta(b) = z_{\partial_0b,\partial_1b}(|b|)\ , 
\qquad b\in\tilde\Si_1(K). 
\]
\end{lemma}
\begin{proof}
Using the definition of transition functions, we have 
\begin{align*}
\ell_{|b|,|b|}\big(z_{\partial_0b,\partial_1b}& (|b|) \big)
  =   \theta^{-1}_{\partial_0b}\, 
 \theta_{\partial_{1}b} \\
& = \theta^{-1}_{\partial_0b}\,
 \theta_{\partial_{1}b}\  
\jmath_{(|b|,\partial_1b)}\  \jmath^{-1}_{(|b|,\partial_1b)}\\
& = \theta^{-1}_{\partial_0b}\, J_{(|b|,\partial_1b)}
 \, \theta_{\partial_{1}b}\  \jmath^{-1}_{(|b|,\partial_1b)}\\
& = \theta^{-1}_{\partial_0b}\, 
J_{(|b|,\partial_0b)} \,  J_{(|b|,\partial_0b)}^{-1}\,  J_{(|b|,\partial_1b)}
 \, \theta_{\partial_{1}b}\ \jmath^{-1}_{(|b|,\partial_1b)}\\
& = \jmath_{(|b|,\partial_0b)}\  \theta^{-1}_{\partial_0b}\, 
     Z(b)
 \,  \theta_{\partial_{1}b} \   \jmath^{-1}_{(|b|,\partial_1b)}\\
& = \jmath_{(|b|,\partial_0b)} \ 
\ell_{\partial_0b,\partial_1b}\big(z_\theta(b)\big) \  
\jmath^{-1}_{(|b|,\partial_1b)}\\
& = \ell_{|b|,|b|}\big( z_\theta(b)\big) \ ,  
\end{align*}
and this completes the proof.
\end{proof}
The mapping $\Gamma_\theta$ defines  a correspondence between 
connections and connection 1--cochains. Moreover, 
as a principal net bundle $\cP$ has a unique flat
connection $Z$, the relation $\Gamma_\theta(Z)=z_\theta$
gives   a correspondence between principal net bundles 
and 1--cocycles. 
We will show that these correspondences are  
equivalences of categories. To this end, we extend this map 
to morphisms. \\
\indent Consider $\hat\cP,\cP\in\sP(K,G)$, and let 
$\hat\theta$ and $\theta$ be, respectively, 
local trivializations of $\hat\cP$ and $\cP$. 
Given $f\in (\hat\cP,\cP)$,   define 
\begin{equation}
\label{Ca:6}
   \big(a, \Gamma_{\theta\hat\theta} (f)_a\,  g\big)  := 
   \big(\theta^{-1}_a\,  f\,  {\hat\theta}_a \big)(a,g) \ ,  
\end{equation}
for any 0--simplex $a$ and for any $g\in G$. 
Clearly, $\Gamma_{\theta \hat\theta}(f):\tilde\Si_0(K)\rightarrow G$.
\begin{proposition}
\label{Ca:7}
Let $\hat\cP,\cP\in\sP(K,G)$,  and let 
$\hat\theta$ and $\theta$ be, respectively, 
local trivializations. Then,
given $\hat U\in \sU(K,\hat\cP)$ and $U\in \sU(K,\cP)$, we have 
\[
 \Gamma_{\theta\hat\theta}(f)
         \in(\hat u_{\hat\theta} , u_\theta), \qquad
         f\in(\hat U,U).
\]
In particular if  $f\in(\hat\cP,\cP)$, then 
$\Gamma_{\theta\, \hat\theta}(f)\in(\hat z_{\hat\theta}, z_\theta)$.
\end{proposition}
\begin{proof}
First of all observe that the  equation (\ref{Ca:6}) is equivalent to 
\begin{equation}
\label{Ca:9}
   \ell_{aa}\big(\Gamma_{\theta\hat\theta} (f)_a\big)  = 
   \theta^{-1}_a \, f\, {\hat\theta}_a \ .
\end{equation}
Now, given a 1-simplex $b$, we have 
\begin{align*}
\ell_{\partial_0b,\partial_1b} 
 \big(\Gamma_{\theta\hat\theta}&(f)_{\partial_0b}\,\hat
 u_{\hat\theta}(b)\big)
 =   \theta^{-1}_{\partial_0b} \,
 f \,  {\hat\theta}_{\partial_0b} \,
         {\hat\theta}^{-1}_{\partial_0b} \, 
 \hat U(b)  \,
 {\hat\theta}_{\partial_1b}\\
&    =  \theta^{-1}_{\partial_0b} \,
 U(b) \,
      \theta_{\partial_1b} \,
      \theta^{-1}_{\partial_1b} \, f \,
 {\hat\theta}_{\partial_1b} =
  \ell_{\partial_0b,\partial_1b}
 \big( u_\theta(b)\, 
  \Gamma_{\theta \hat\theta}(f)_{\partial_1b}\big),
\end{align*}
and this completes the proof.
\end{proof}
This result provides the formula for  changing  
a local trivialization. In fact let $\theta$ 
and $\phi$ be two local trivialization of the bundle $\cP$. Then 
\begin{equation}
\label{Ca:10}
\Gamma_{\theta\phi}(1_\cP)\in 
 (u_\phi,u_\theta), \qquad U\in\sU(K,\cP).
\end{equation}
Hence, changing the  local trivializations leads to 
equivalent connection 1--cochains. \smallskip

We now show that $\Gamma_\theta$ 
is invertible.  To this end, given
$u\in\rU^1(K, z_\theta)$, $\hat u\in\rU^1(K,{\hat z}_{\hat\theta})$, 
and $\mathrm{f}\in (u,\hat u)$, define 
\begin{equation}
\label{Ca:11}
\begin{array}{rcll}
 \Upsilon_\theta(u)(b) (\psi) & :=&  \big(\theta_{\partial_0b}\,
 \ell_{\partial_0b,\partial_1b}(u(b)) \,
 \theta^{-1}_{\partial_1b}\big)(\psi) \ , & 
  \psi\in\pi_1^{-1}(\partial_1b) \ ,    \\[5pt]
 \Upsilon_{\theta\hat\theta}(\mathrm{f})(\psi) & :=& 
\big(\theta_a\, \ell_{a,a}(\mathrm{f}_a) \, {\hat\theta}^{-1}_a \big) (\psi)\ , & 
 \psi\in\pi_1^{-1}(a)\ .
\end{array}
\end{equation}
for any $b\in\tilde\Si_1(K)$ and $a\in\tilde\Si_0(K)$.
\begin{proposition}
\label{Ca:12}
Let $\cP,\hat\cP$ be  principal net bundles with local trivializations
$\theta,\hat\theta$, respectively. Then 
\begin{itemize}
\item[(i)] $\Upsilon_\theta:\rU^1(K,z_\theta)\to 
            \sU(K,\cP)$  is the  inverse   of 
             $\Gamma_\theta$.
\item[(ii)] Given   $u\in\rU^1(K,z_\theta)$ and 
          $\hat u\in\rU^1(K,{\hat z}_{\hat\theta})$, then   
          $\Upsilon_{\theta\hat\theta}: (\hat u,u)
         \to(\Upsilon_{\hat\theta}(\hat u),
                                      \Upsilon_\theta(u))$
          is the inverse  of $\Gamma_{\theta\hat\theta}$
\end{itemize}
\end{proposition}
\begin{proof}
$(i)$ Clearly $\Upsilon_\theta(u)(b)$ is a bijection 
from $\pi^{-1}(\partial_1b)$ into $\pi^{-1}(\partial_0b)$. 
Given the reverse $\overline{b}$ of $b$ we have 
\[
 \Upsilon_\theta(u)(\overline{b})  =  \theta_{\partial_0\overline{b}}\
 \ell_{\partial_0\overline{b},\partial_1\overline{b}}(u(\overline{b}))
 \  \theta^{-1}_{\partial_1\overline{b}}
 =  \theta_{\partial_1b}\,
 \ell_{\partial_1b,\partial_0b}(u(b)^{-1})  \  
 \theta^{-1}_{\partial_0b} =  \Upsilon_\theta(u)(b)^{-1}.
\]
Let $b$ be a 1--simplex of the nerve. Since 
$u(b)=z_\theta(b)$, by (\ref{Ca:3}),  we have
\[
\Upsilon_\theta(u)(b)  = 
  \theta_{\partial_{0}b} \,
 \ell_{\partial_{0b},\partial_{1}b}(z_\theta(b))
 \, \theta^{-1}_{\partial_{1}b}
 = \theta_{\partial_{0}b} \, \theta^{-1}_{\partial_{0}b} \, Z(b) \, 
     \theta_{\partial_{1}b}\, \theta^{-1}_{\partial_{1}b} 
  = J_{\partial_0b,\partial_1b}.
\]
Hence $\Upsilon_\theta(u)$ is a connection of $\cP$.   
The identities $\Gamma_\theta(\Upsilon_\theta(u)) = u$ and 
$\Upsilon_\theta(\Gamma_\theta(U)) =  U$ follow
straightforwardly from the definitions of $\Gamma_\theta$ and 
$\Upsilon_\theta$. $(ii)$ 
Given a 1--simplex $b$, we have  
\begin{align*}
 \Upsilon_{\theta\hat \theta}(\mathrm{f}) \, \Upsilon_{\hat\theta}(u)(b) & = 
\theta_{\partial_0b} \,  
\ell_{\partial_0b\partial_0b}\big(\mathrm{f}_{\partial_0b}\big) \, 
\ell_{\partial_{0}b,\partial_{1}b}\big(u(b)\big)
 \, {\theta_1}^{-1}_{\partial_{1}b} \\
& = 
\theta_{\partial_0b} \,  
\ell_{\partial_0b\partial_1b}\big(\mathrm{f}_{\partial_0b}\,  u(b)\big)
 \, {\hat\theta}^{-1}_{\partial_{1}b} \\
& = 
\theta_{\partial_0b} \,  
\ell_{\partial_0b\partial_1b}\big(u(b)\, \mathrm{f}_{\partial_1b}\big)
 \, {\hat\theta}^{-1}_{\partial_{1}b} \\
& =  \Upsilon_{\theta}(u)(b)\, \Upsilon_{\theta\hat\theta}(\mathrm{f}) \ 
\end{align*}
This, in particular,  implies that 
$\Upsilon_{\theta\hat\theta}(f)\, {\hat J}_{b} = J_{b}\, f$ 
for any $b\in \Si_1(K)$,
as any connection coincides with the flat connection when restricted 
to the nerve. Equivariance is obvious.  
\end{proof}
The previous results point to an equivalence 
between the category  of  connections 
of a principal net bundle and  the category of connection $1$--cochains 
of the corresponding bundle cocycle.
\begin{theorem}
\label{Ca:13}
Given a principal net bundle $\cP$ and a local trivialization
$\theta$. Then the categories $\sU(K,\cP)$ and 
$\rU^1(K,z_\theta)$ are isomorphic. 
\end{theorem}
\begin{proof}
Using Propositions \ref{Ca:2} and \ref{Ca:5}, it is easily seen 
that the mappings 
$\sU(K,\cP)\ni U\to  u_\theta\in\rU^1(K,z_\theta)$ and 
$(U,\tilde U)\ni f\to \Gamma_{\theta\theta}(f)\in (u_\theta,
{\tilde u}_\theta)$,
define  a covariant 
functor from $\rU^1(K,z_\theta)$ to  $\sU(K,\cP)$.
Conversely,  by Proposition \ref{Ca:12}, 
the mappings $\rU^1(K,z_\theta)\ni u\to
\Upsilon_\theta(u)\in\sU(K,\cP)$ and 
$(u,\tilde u)\ni \mathrm{f}\to  
 \Upsilon_{\theta\theta}(\mathrm{f})\in
(\Upsilon_{\theta}(u),\Upsilon_{\theta}(\tilde u))$ 
define a covariant functor from 
$\rU^1(K,z_\theta)$ to $\sU(K,\cP)$.  
Proposition \ref{Ca:12} implies that these two
functors are isomorphisms of categories.
\end{proof}
%
%
\subsection{General equivalence}
\label{Cb}
We want  to prove
that any $1$--cocycle $z$ gives rise to a 
principal net bundle and any connection $1$--cochain 
$u\in \rU^1(K,z)$ to a connection 
of the principal net bundle associated with $z$. We will prove this
by a series  of results. The first one imitates 
the reconstruction of a principal bundle from its transition functions.
\begin{proposition}
\label{Cb:1}
Given  any $1$--cocycle $z\in\rZ^1(K,G)$ there is a
principal net bundle $\cP_z$, with structure group $G$, 
and  a local trivialization    $\theta_{z}$
such that  $\Gamma_{\theta_z}(Z) = z$, where $Z$ is the flat
connection of $\cP_z$.
\end{proposition}
\begin{proof}
Given $z\in\rZ^1(K,G)$, for any pair $a,\tilde a$, define
\[
 z_{a\tilde a}(o)  := z(o;a,\tilde a), \qquad o\in V_{a\tilde a}
\]
where $(o;a,\tilde a)$ is the 1--simplex with support $o$, 1--face 
$\tilde a$ and 0--face  $a$. It is easily seen that the functions 
$z_{a\tilde a}: V_{a\tilde a}\rightarrow G$ satisfy 
all the properties of transition
functions (Lemma  \ref{B:11}). Let 
\[
  P := \bigcup_{a\in\tilde\Si_0(K)} V_a\times  G\times \{a\}
\]
Any element of $P$ is of the form $(o,g,a)$ where 
$a\in\tilde\Si_0(K)$, $o\in V_a$ and $g\in G$. Equip 
$P$ with the  equivalence relation 
\[
 (o,g,a) \sim_z (o_1,g_1,a_1) \ \iff \ o=o_1\in V_{aa_1}
   \mbox{ and }   g = z_{aa_1}(o)g_1.
\] 
Denote the equivalence 
class associated with the element $(o,g,a)$ by $[o,g,a]_z$, and  define
\[
 P_z 
  := \{ [o,g,a]_{z} \ | \ a\in\tilde\Si_0(K), \ o\in V_a, \ g\in G \}
\]
There is a surjective map
$\pi_z: P_z\to  K$ defined by 
$\pi_{z}([o,g,a]_z) := o$ with 
$[o,g,a]_z\in P_z$. 
Given  $b\in \Si_1(K)$, define 
\[
 {J_z}_{b}[\partial_1b,g,a]_z := [\partial_0b,g,a]_z.
\]
This definition is well posed. In fact, if $(\partial_1b,g,a)\sim_z 
(\partial_1b,g_1,a_1)$, then $g=z_{aa_1}(\partial_1b)\cdot g_1$.  
Since $\partial_1b\leq\partial_0b$
we have $z_{aa_1}(\partial_1b) = z_{aa_1}(\partial_0b)$. Hence 
$(\partial_0b,g,a)\sim_z  (\partial_0b, g_1,a_1)$. 
${J_z}_{b}$ is a bijection and the fibres are isomorphic
to  $G$. Define 
\[
  {R_z}_h [o,g,a]_z  := [o, gh,a]_z, \qquad h\in G.
\]
This is  a free right action on $P_z$  because 
$[o,g,a]_z=[o,gh,a]_z$ if, and only if, $g = gh$ 
and is transitive on the fibres.
Given $b\in \Si_1(K)$ we have 
\[
 {J_z}_{b}\, {R_z}_h  [\partial_1b,g,a]_z 
  = [\partial_0b, gh,a]_z 
   = 
   {R_z}_h[\partial_0b,g,a]_z = {R_z}_h\, {J_z}_{b}
 [\partial_1b,g,a]_z.
\]
So the data
$\cP_z := (P_z, \pi_z, J_z, G,K)$
is a principal net bundle over $K$ with structure group $G$. 
Given a $0$--simplex $a$, let  
\[
  {\theta_z}_a(o,g)  := {J_z}_{(o,a)}([a,e,a]_z)\cdot g,
 \qquad (o,g)\in V_a\times G.
\]
Observe that 
\[
 {\theta_z}^{-1}_{a_1}\, {\theta_z}_a(o,g)  = 
  {\theta_z}^{-1}_{a_1} [o,g,a]_z
  = {\theta_z}^{-1}_{a_1} [o,z_{aa_1}(o)g,a_1]_z 
  = (o,z_{aa_1}(o) g).
\]
The proof then follows by  Lemma \ref{Ca:5}.
\end{proof}
Thus the principal 
net bundle $\cP_z$ has bundle cocycle $z$. 
Now, since any connection 1--cochain induces 
a 1--cocycle, we expect that any connection 1--cochain
comes from   a connection 
on a principal net bundle. This is the content 
of the next result. \\
\indent For any $\cP\in\sP(K,G)$, choose 
a local trivialization $\theta$ of $\cP$ and denote
such a \emph{choice} by the symbol
$\ute$.  
\begin{theorem}
\label{Cb:2}
The following assertions hold: 
\begin{itemize}
\item[(i)] The categories $\rZ^1(K,G)$ and $\sP(K,G)$ are equivalent;
\item[(ii)] The categories $\rU^1(K,G)$  and $\sU(K,G)$ are equivalent.
\end{itemize}
\end{theorem}
\begin{proof}
We first prove the assertion $(ii)$. Define 
\begin{equation}
\begin{array}{ll}
 \Gamma_{\ute}(U)  := \Gamma_{\theta}(U) \ , & U\in \sU(K,\cP)\ ; \\[3pt] 
 \Gamma_{\ute}(f)  := \Gamma_{\theta\hat\theta}(f) \ , & 
  f\in (\hat U,U),
\end{array}
\end{equation}
where  $U\in \sU(K,\cP)$, and $U_1\in \sU(K,\hat \cP)$.
It easily follows from  Propositions \ref{Ca:2} and \ref{Ca:7}  that 
$\Gamma_{\ute}:\sU(K,G)\to \rU^1(K,G)$ is a covariant functor.
Conversely, observe  that  $\rU^1(K,G)$ is a disjoint union 
of the $\rU^1(K,z)$ as $z$ varies in $\rZ^1(K,G)$ (see Appendix), and define 
\begin{equation}
\begin{array}{ll}
  \Upsilon(u)  := \Upsilon_{\theta_z}(u) \ , & u\in\rU^1(K,z);\\[3pt]
  \Upsilon(\mathrm{f})  := 
  \Upsilon_{\theta_{\hat z}\theta_z}(\mathrm{f}) \ , &  
  \mathrm{f}\in(u,\hat u),
\end{array}
\end{equation}
where  $u\in\rU^1(K,z)$ and $\tilde u\in\rU^1(\tilde z)$, and 
$\theta_z$ is the local trivialization of $\cP_z$ 
in Proposition \ref{Cb:1}. By
Proposition \ref{Ca:12},  
$\Upsilon$ is a covariant functor. \\
\indent We now prove that the pair $\Gamma_{\ute}$, $\Upsilon$ is 
an equivalence of categories.
To this end, consider a connection  $U$ on $\cP$  and the connection 
$\Upsilon\Gamma_{\ute}(U) = \Upsilon_{\theta_{z_\theta}}(u_\theta)$
defined on the principal 
net bundle $\cP_{z_\theta}$ reconstructed 
from the bundle cocycle $z_\theta$. 
Define 
\[
 x(U)[o,\, g, \, a]_{z_\theta}  := 
     \theta_a(o,g),  \qquad (o,\, g)\in V_a\times G \ .
\]
This definition is well posed. In fact 
if $(o,\, g,\, a)\sim_{z_\theta}(o,\, g_1,\,
a_1)$, that is, if 
$g = {z_\theta}_{aa_1}(o)g_1$, then 
$\theta_a(o, g)  =  
\theta_a(o, {z_\theta}_{aa_1}(o) g_1)  =                               
\theta_a\, \theta^{-1}_a  \,
\theta_{a_1}
(o,g_1)
=  \theta_{a_1}(o,g_1)$, 
where $\{{z_\theta}_{aa_1}\}$ are the transition functions 
associated with $z_\theta$. The definitions of 
$\Upsilon_{\theta_{z_\theta}}$ and $\Gamma_{\theta}(U)$ imply 
that  $x$ is a natural isomorphism
between 
$\Upsilon \,  \Gamma_{\ute}$ and $1_{\sU(K,G)}$.  Conversely 
given a $z\in\rZ^1(K,G)$, define  
\[
 y(u)  := \Gamma_{\theta_z \theta'}(\mathrm{I}_{\cP_z}), \qquad 
           u\in\rU^1(K,z),
\]
where $\theta'$ denotes the local trivialization of $\cP_z$ associated
with $\ute$, and $\theta_z$ that 
defined in Proposition \ref{Cb:1}. Proposition \ref{Ca:7} implies
that $y$ is a natural isomorphism 
between $\Gamma_{\ute}\,\Upsilon$ and 
$1_{\rU^1(K,G)}$.
$(i)$ follows from $(ii)$ by observing that 
$\Gamma_{\ute}:\sU_f(K,G)\to \rZ^1(K,G)$ and 
$\Upsilon:\rZ^1(K,G)\to \sU_f(K,G)$ and  that $\sU_f(K,G)$ is isomorphic
to $\sP(K,G)$ (Lemma \ref{B:7}). 
\end{proof}

%
%
\section{Curvature and reduction}
\label{D}
The equivalence between the geometrical and the cohomological description 
of principal bundles over posets allows us to analyze 
connections further. In particular,  
we will discuss: 
the curvature of a connection, relating it to  
flatness; the existence of non-flat connections; the reduction 
of principal net bundles and connections.\bigskip
\subsection{Curvature}
\label{Da}

We now introduce the curvature of 
a connection and use the notation of
(\ref{Ca:4}) to indicate the connection cochain and the bundle cocycle. 
\begin{definition}
\label{Da:1}
Let $\cP\in\sP(K,G)$ 
have a local trivialization $\theta$. 
The \textbf{curvature cochain} of a connection $U$ on $\cP$ 
is the $2$--coboundary 
$\mathrm{d}u_\theta\in\mathrm{B}^2(K,G)$ of the connection cochain
$u_\theta\in\rU^1(K,G)$, namely 
\[
\begin{array}{lcll}
(\mathrm{d} u_{\theta})_1(b) &  := & 
(\io, \  \mathrm{ad}(u_{\theta}(b)), & b\in\tilde\Si_1(K) \ , \\
(\mathrm{d}u_{\theta})_2(c)  &  := & 
 (w_{u_{\theta}}(c), \ \mathrm{ad}(u_{\theta}(\partial_1c)),
 & c\in\tilde\Si_2(K) \ ,
\end{array}
\]
where $w_{u_{\theta}}: \tilde\Si_2(K)\to G$ is defined by 
\[
 w_{u_{\theta}}(c)  := u_{\theta}(\partial_0c)\, 
 u_{\theta}(\partial_2c)\, 
 u_{\theta}(\partial_1c)^{-1}, \qquad c\in\tilde\Si_2(K) \ ,
\]
and $\mathrm{ad}(g)$ denotes the inner automorphism of $G$ 
defined by $g\in G$.
\end{definition}
Some  observations are in order. 
\emph{First}, this definition does not depend on the choice 
of local trivialization. To be precise,  
if $\theta_1$ is another  local trivialization of 
$\cP$, the two coboundaries $\mathrm{d}u_\theta$ and 
$\mathrm{d}u_{\theta_1}$  are equivalent objects 
in the category of 2--cochains (see \cite{RR06,Rob79}).
However, we prefer not to verify this as it  would involve 
the heavy machinery of  non-Abelian cohomology. 
\emph{Secondly}, in this framework, the Bianchi identity 
for a connection 1--cochain $u$ 
corresponds to the 2--cocycle identity, 
namely 
\begin{equation}
\label{Da:2}
 w_u(\partial_0d) \ w_u(\partial_2d)  = 
 \mathrm{ad}(u(\partial_{01}d))\big(w_u(\partial_3d)\big) \ 
                           w_u(\partial_1d), 
\end{equation}
for any 3--simplex $d$. \emph{Thirdly}, in \cite{RR06} a connection $u$ 
is defined to be flat whenever its curvature is trivial. 
This amounts to saying that $w_u(c) =e$ 
for any 2--simplex $c$. Now as the mapping $\Gamma_\theta$ is invertible, 
a \emph{connection $U$ is flat 
if, and only if, its curvature is trivial}. \smallskip

We now draw on two consequence of this definition of curvature 
and of the equivalence between the geometrical and the cohomological
description of principal bundles over posets.

\begin{corollary}
\label{Da:3}
There is, up to equivalence, 
    a 1-1 correspondence between flat connections of principal net
    bundles   over $K$ with structure
    group $G$  and   group homomorphisms 
    of the fundamental group $\pi_1(K)$ of the  poset  
    with values in the group $G$. 
\end{corollary}
\begin{proof}
By  Theorem \ref{Cb:2}  the proof follows from  
\cite[Corollary 4.6]{RR06}. 
\end{proof}
This is the poset version of a classical result of the theory 
of fibre bundles  over manifolds (see \cite{KN}). 
Note that this corollary and Theorem \ref{Cb:2}$(i)$ 
imply  that principal net bundles over a simply connected poset 
are trivial. The existence of nonflat connections is the 
content of the next result.
\begin{corollary}
\label{Da:4}
Assume that $K$ is a pathwise connected  but not totally directed 
    poset, and let $G$ be a nontrivial group.  Then, the set of
    nonflat connections of a principal net bundle
    $\cP\in\sP(K,G)$ is never empty.
\end{corollary}
\begin{proof}
By Theorem \ref{Cb:2}, the proof follows from  
\cite[Theorem 4.25]{RR06}. 
\end{proof}

%
%
\subsection{Reduction theory}
\label{Db}
The cohomological representation allows us 
to compare principal bundles having different structure groups.     
Any group homomorphism $\gamma:H\to G$ defines 
a  covariant functor 
$\gamma\circ:\rC^1(K,H)\to \rC^1(K,G)$ (see Appendix). This  
functor maps  1--cocycles and connection
1--cochains with values in $H$
into  1--cocycles and connection 1--cochains with values
in $G$. 
Moreover ${\ga\circ}$ turns out to be an isomorphism whenever $\ga$ 
is a group isomorphism. So using the functors 
$\Upsilon$ and 
$\Gamma_{\ute}$ introduced in the proof 
of Theorem \ref{Cb:2}, we have that 
\[
  \Upsilon \, {\gamma\circ} \, 
   \Gamma_{\ute}: \sU(K,H)\to \sU(K,G)
\]
and, clearly, 
\[
 \Upsilon\, {\gamma\circ} \, \Gamma_{\ute}: \sP(K,H)\to \sP(K,G)
\]
are covariant functors and even equivalences of categories  
when $\ga\circ$ is a group isomorphism.\\ 
\indent We now define 
reducibility both for connections and principal net 
bundles.
\begin{definition}
\label{Db:1}
A \textbf{connection} $U$ on a principal net bundle 
$\cP\in\sP(K,G)$ is said to be \textbf{reducible},
if there is a proper subgroup $H\subset G$, a 
principal net bundle $\tilde \cP\in\sP(K,H)$ and 
a connection $\tilde U\in \sU(K,\tilde \cP)$ such that 
\begin{equation}
\Upsilon_{\theta_{{\tilde z}_\phi}} \, 
{\io_{\sst{G,H}}}\, 
\Gamma_{\phi}(\tilde U) \cong U \ .
\end{equation}
for a local trivialization $\phi$ of $\tilde \cP$.
A \textbf{principal net bundle} $ \cP$ is \textbf{reducible} 
if its flat connection is reducible. 
\end{definition}
In this definition $\io_{\sst{G,H}}:H\to G$
is the inclusion mapping. So, 
recalling the notation  (\ref{Ca:4}),
$\Upsilon_{\theta_{{\tilde z}_{\phi}}} \, 
{\io_{\sst{G,H}}}\, 
\Gamma_{\phi}(\tilde U)$  is the reconstruction 
of a connection from the connection 1--cochain 
${\tilde u}_{\phi}\in \rU^{1}(K,H)$ considered as 
taking values in the larger group $G$.
$\theta_{\tilde z_{\phi}}$ is the local trivialization 
of the principal net bundle reconstructed from the bundle cocycle 
${\tilde z}_{\phi}$ 
(see Proposition \ref{Cb:1}) considered as a cocycle with
values in $G$.
Moreover, the bundle cocycle ${\tilde z}_\phi$ is equal to 
the 1--cocycle induced by $\tilde u_\phi$ (Proposition \ref{Ca:2}).
Finally, by Theorem \ref{Cb:2} this definition 
is independent of the local trivialization. \smallskip

We now derive an analogue of the Ambrose-Singer theorem 
for principal bundles and connections over posets. 
This has already been done in the cohomological approach \cite[Theorem 4.28]{RR06}.  
The next two lemmas are needed for extending this result 
to fibre bundles.
\begin{lemma} 
\label{Db:2}
A connection 
$U$ on $\cP\in\sP(K,G)$ is reducible if, and only if, 
the connection cochain $u_\theta$ of a local trivialization 
$\theta$ is reducible.   
\end{lemma}
\begin{proof}
$(\Rightarrow)$ Assume that $u_\theta$  is reducible.
Then there is a proper subgroup 
$H$ of $G$ and a connection 1--cochain $ v\in\rU^1(K,H)$ 
such that $v \cong u_\theta$ in the category $\rU^1(K,G)$. 
By virtue of Theorem \ref{Cb:2}, 
there is a connection $\tilde U\in\sU(K,\tilde \cP)$ with 
$\tilde \cP\in\sP(K,H)$ such that $\tilde u_{\tilde\theta} = v$
for a local trivialization $\tilde\theta$ of $\tilde \cP$. Hence 
${\tilde u}_{\tilde\theta}\cong u_\theta$ in $\rU^1(K,G)$.
The  corresponding bundle cocycles
${\tilde z}_{\tilde\theta}$ and $z_{\theta}$ 
are then equivalent in $\rZ^1(K,G)$ \cite[Lemma 4.14]{RR06}. 
By Proposition \ref{Ca:12}$(ii)$ we have
\[ 
\Upsilon_{\theta_{{\tilde z}_{\tilde\theta}}}
({\tilde u}_{\tilde\theta}) \cong 
\Upsilon_{\theta_{z_{\theta}}} (u_\theta) = U.
\]
This completes the proof.
The implication $(\Leftarrow)$ follows similarly.
\end{proof}
We next relate the holonomy groups of a 
connection to those of the corresponding connection cochain.
The holonomy group $H_u(a)$ and the restricted holonomy 
group $H^0_u(a)$, based on a 0--simplex $a$,
of a connection 1--cochain $u$ of $\rU^1(K,G)$ are  defined,
respectively,  by 
$H_u(a) =  \big\{  u(p)\in G \ |  \ \partial_0p=a=\partial_1p\big\}$ and 
$H^0_u(a) =  \big\{  u(p)\in G \ | \ \partial_0p=a=\partial_1p, 
\ p\sim\si_0a\big\}$ 
(see \cite{RR06}). Then we have the following 
\begin{lemma}
\label{Db:3}
Let $U\in\sU(\cP)$ and $\psi\in\cP$ with $\pi(\psi)=a$.  
Given  a local trivialization  $\theta$ of $\cP$, 
$H_U(\psi)$ and  $H_{u_\theta}(a)$  are conjugate subgroups of $G$.  
The same assertion holds for the restricted holonomy groups. 
\end{lemma}
\begin{proof}
If $g\in H_U(\psi)$, then there is a loop $p$ such that 
$\partial_0p=a=\partial_1p$ and $R_g(\psi) = U(p)\psi$. Now, let  
$(a,g_\psi):= \theta_a^{-1}\psi$, then 
\begin{align*}
(a,g_\psi g) & = r(g)\, \theta^{-1}_a \psi =  \theta_a^{-1} R_g\psi
              = \theta^{-1}_a\, U(p)\psi\\
           &   = \theta^{-1}_a\, U(b_n)\cdots U(b_1)\psi\\
           &   = \theta^{-1}_a\, U(b_n)\theta_{\partial_1b_n} \theta^{-1}_{\partial_1b_n}   \cdots  U(b_2)\theta^{-1}_{\partial_0b_1} 
\theta_{\partial_0b_1} U(b_1)\theta_a \theta^{-1}_a\psi\\ 
           &   = \ell_{a,\partial_1b_n}(u_\theta(b_n))\cdots 
                  \ell_{\partial_0b_1,a}(u_\theta(b_1)) \theta^{-1}_a\psi\\ 
&            = \ell_{a,a}(u_\theta(p)) \theta^{-1}_a\psi \\
           & = (a, \ u_\theta(p) g_\psi)
\end{align*} 
The mapping 
$H_U(\psi)\ni g\to  g_\psi g g_\psi^{-1}\in H_{u_\theta}(a)$
is clearly a group isomorphism. 
\end{proof}
This result  and \cite[Lemma 4.27]{RR06} lead  to the
following conclusions: the holonomy groups are subgroups of $G$; 
the restricted group  is a normal subgroup of the unrestricted group;
a change of the base $\psi$ leads to conjugate holonomy groups
and equivalent connections lead to isomorphic holonomy groups.\smallskip 
 
Finally, we have the analogue of the Ambrose-Singer theorem for
principal net bundles.
\begin{theorem}
\label{Db:4}
Given $\cP\in\sP(K,G)$, let $\psi\in \cP$.  
Let  $U\in \sU(K,\cP)$. Then 
\begin{itemize}
\item[(i)] $\cP$  is reducible to a principal net bundle $\tilde \cP\in
\sP(K,H_U(\psi))$. 
\item[(ii)] $U$ is reducible to a connection $\tilde U\in
\sU(K,\tilde \cP)$.
\end{itemize}
\end{theorem}
\begin{proof}
The proof follows from the previous lemmas 
and  \cite[Theorem 4.28]{RR06}.
\end{proof}

\section{\v Cech cohomology and flat bundles}
\label{E}

In the previous sections, we introduced the transition functions
of a (principal) net bundle, and verified that they
satisfy cocycle identities analogous to those 
encountered in the \v Cech cohomology of a topological space.
This allows us to define the 
{\em \v Cech cohomology} of  a poset,  showing  it to be equivalent to the 
net cohomology of the poset.
Then, we specialize our discussion to the 
poset of open, contractible subsets of a manifold. In this case, 
the above constructions yield 
the locally constant cohomology of the manifold,
which, as is well known, describes the category of flat
bundles (see  \cite[I.2]{Kob}).\bigskip

A \emph{ \v Cech cocycle} of  $K$ with values in $G$,  is  a family 
$\xi :=$ $\left\{ \xi_{aa'} \right\}$ of locally constant 
maps $\xi_{aa'} : V_{aa'} \to G$ satisfying the cocycle relations
\[
\xi_{\hat a\tilde a} (o) \xi_{\tilde aa} (o) =
\xi_{\hat aa} (o), \qquad o\in V_{a\tilde a\hat a}.
\]
A cocycle $\xi'$ is said to be \emph{equivalent} 
to $\xi$ if there is a family 
$u :=$ $\left\{ u_a \right\}$ 
of locally constant maps $u_a : V_a \to G$ such 
that $\xi'_{a\tilde a} (o) u_{\tilde a} (o) =$
$u_a (o)\xi_{a\tilde a} (o)$,
$o\in V_{a\tilde a}$. We denote the set of \v Cech cocycles and their 
equivalence classes  by $\rZ_c^1 ( K , G )$ and $\rH_c^1 ( K , G)$,
respectively. In the sequel, by a slight abuse of notation,  we 
use the same symbols to denote cocycles and the corresponding 
equivalence classes.\smallskip

Let us now consider the dual poset $K^\circ$ and its  
fundamental covering $\cV^\circ_0$.
Note  that $o\in V^\circ_a$ if and only if $a\leq^\circ o$,
or equivalently if $o\leq a$. The aim is to establish 
a correspondence between $\rH^1(K,G)$, $\rH^1_{\bf c}(K^\circ,G)$ and 
$\rH^1(K^\circ,G)$. Given $z\in \rZ^1(K,G)$, define 
\begin{equation}
 z^{\bf c}_{a\tilde a}(o) :=  z(a;a,o) \, z(\tilde a;\tilde a, o)^{-1},
 \qquad o\in V^\circ_{a\tilde a}.
\end{equation}
The definition is well posed since $a,\tilde a\leq^\circ o$ 
implies that  $(a;a,o)$ and $(\tilde a;\tilde a, o)$ are 
1--simplices of the nerve of $K$. 
Following the reasoning  of \cite[Lemma 4.10]{RR06}, 
$z^{\bf c}$ turns out to be  
the unique \v Cech cocycle of $\rZ^1_{\bf c}(K^\circ,G)$ 
satisfying the equation 
$z^{\bf c}_{\partial_1b,\partial_0b}(\partial_0b)= 
            z(\partial_1b;\partial_1b, \partial_0b)$ for
any 1--simplex of the nerve of $K^\circ$. Thus we have an injective map 
\begin{equation}
\label{eq_c}
\rH^1 (K,G)  \to  \rH_c^1(K^\circ,G), \ \ \
z \mapsto z^{\bf c},
\end{equation}
since any $1$--cocycle is uniquely determined by its values on $\Sigma_1(K)$.
By applying \cite[Theorem 4.12]{RR06}, we also construct a 
bijective map
\begin{equation}
\label{eq_n}
\rH_c^1 (K^\circ,G)  \to  \rH^1(K^\circ,G), \ \ \ 
\xi \mapsto \xi^{\bf n}, 
\end{equation}
where $\xi^{\bf n} (b) :=$  $\xi_{\partial_0b , \partial_1b} (|b|)$, 
$b \in \tilde\Si_1(K^\circ)$ and  $\xi \in \rZ_c^1(K^\circ,G)$. 
If we compose (\ref{eq_c})
and (\ref{eq_n}), then we obtain a injective  map
\begin{equation}
\label{eq_circ}
\rH^1 ( K,G ) \to \rH^1 (K^\circ,G), \ \ \ 
z \mapsto z^\circ := (z^{\bf c})^{\bf n}\ .
\end{equation}
We now observe that  $z^{\circ \circ} = z$, 
$z \in \rH^1 ( K,G )$, implying that (\ref{eq_circ}) and hence 
(\ref{eq_c}) are bijective. First note that 
\[
  z^\circ (b) = z(\partial_0b;\partial_0b, |b|)\, 
                z(\partial_1b;\partial_1b,|b|)^{-1},
\qquad  b\in\tilde\Si_1(K^\circ).
\]
So, given $b\in\tilde\Si_1(K)$ we have 
\begin{align*}
  z^{\circ\circ} (b) & = z^{\circ }(\partial_0b;\partial_0b, |b|)\, 
                z^{\circ}(\partial_1b;\partial_1b,|b|)^{-1}\\
        & = z(\partial_0b; \partial_0b, \partial_0b)
         \, z(|b|;|b|,\partial_0b)^{-1}\,
         (z(\partial_1b; \partial_1b, \partial_1b)
         \, z(|b|;|b|,\partial_1b)^{-1})^{-1}\\
        & = z(\si_0\partial_0b)
         \, z(|b|;|b|,\partial_0b)^{-1}\,
         (z(\si_0\partial_1b)
         \, z(|b|;|b|,\partial_1b)^{-1})^{-1}\\
        & =  z(|b|;\partial_0b, |b|) \, z(|b|;|b|,\partial_1b)\\
        & =  z(b),
\end{align*}
where we have applied the 1--cocycle identity to the last equality. 
The following result summarizes the previous considerations 
and incorporates an immediate consequence of Theorem \ref{Cb:2}.
\begin{theorem}
\label{E:4.0}
The maps $z \mapsto z^{\bf c}$, $\xi \mapsto \xi^{\bf n}$ induce 
one-to-one correspondences 
\[
\rH^1 (K,G) \to \rH_c^1 (K^\circ,G)
\ \ , \ \
\rH_c^1 (K^\circ,G) \to \rH^1 (K^\circ,G)
\ ,
\]
defined so that, exchanging the role of $K$ and $K^\circ$,
the diagram
\begin{equation}
\label{eq_cdd}
\xymatrix{
    \rH^1 ( K,G )
    \ar[r]
 &  \rH_c^1 ( K^\circ,G )
    \ar[d]
 \\ \rH_c^1 ( K,G )
    \ar[u]
 &  \rH^1 ( K^\circ,G )
    \ar[l]
}
\end{equation}
commutes. Moreover, the map (\ref{eq_circ}) induces an isomorphism
of categories from $\sP(K,G)$ to $\sP(K^\circ,G)$.
\end{theorem}
%


We now specialize taking our poset $K$ to be a base for the topology of a 
manifold ordered under inclusion. Thus we can
relate (principal) net bundles to
well-known geometrical constructions, involving locally constant
cohomology and flat bundles.

Let $M$ be a paracompact, arcwise connected and locally contractible space. 
We denote by $K$ the poset whose elements are the open, contractible subsets 
of $M$.
A \emph{locally constant cocycle} is given by a pair $( A , f )$, 
where $A \subseteq K$ is a locally finite open cover of $M$, and 
\[
f := \left\{ f_{a'a} : a \cap a' \to G \ , \ a,a' \in A \right\}
\]
is a family of locally constant maps satisfying the cocycle relations
\[
f_{a''a'} (x) f_{a'a} (x) = f_{a''a} (x)
\ ,
\]
$x \in a \cap a' \cap a''$.
A locally constant cocycle $( \tilde{A} , \tilde{f} )$ is said to be 
\emph{equivalent}  to $( A , f )$ if 
there are locally constant maps $v_{\tilde{a}a} : \tilde{a} \cap a \to G$, 
$a \in A$, $\tilde{a} \in \tilde{A}$, 
such that 
\[
v_{a \tilde{a}} (x) \tilde{f}_{ \tilde{a} \tilde{a}'} (x) =
f_{aa'} (x) v_{a'\tilde{a}'} (x),
\]
where  $x \in a \cap a' \cap \tilde{a} \cap \tilde{a}'$, 
$a,a' \in A$, $\tilde{a} , \tilde{a}' \in \tilde{A}$. Note that in locally 
constant cohomology we work with elements of $A$ rather than with
$V_a$, $a \in K$. This makes locally
constant cocycles more manageable than cocycles 
defined over $\tilde\Sigma_0(K)$.

The set of locally constant cocycles is denoted by $\rZ_{lc}^1 (M,G)$
and their equivalence classes by
$\rH_{lc}^1 ( M , G )$
and called the (first) {\em locally constant cohomology} of $M$.\smallskip

Now, if we endow $G$ with the discrete topology, 
then $\rH_{lc}^1(M,G)$ coincides with the usual 
cohomology set classifying principal $G$-bundles over 
$M$ (see Remark \ref{rem_h1lc} below). In this way we obtain the following
\begin{theorem}
\label{thm_hmor}
Let $M$ be a paracompact, arcwise connected, locally contractible
space, and $G$ a group. Then, there are isomorphisms
\[
\rH^1(K,G) 
\ \simeq \
\dot{\rH}{\mathrm{om}} (\pi_1(K),G)
\ \simeq \
\dot{\rH}{\mathrm{om}} (\pi_1(M),G)
\ \simeq \
\rH_{lc}^1(M,G)
\ ,
\]

where $\dot{\rH}{\mathrm{om}} (-,G)$ 
denotes the set of $G$-valued morphisms 
modulo inner automorphisms of $G$.
\end{theorem}
\begin{proof}
The first three isomorphisms have been already established
\cite{Ruz05} (see also \cite{RR06}). The last isomorphism, instead, 
follows  using the classical machinery developed in 
\cite[I.13]{Ste} for totally disconnected groups.
In fact, locally constant cocycles $(A,f) \in\rZ_{lc}^1(M,G)$ 
are in one-to-one correspondence with morphisms 
$\chi_{(A,f)}$ from  $\pi_1(M)$ into  $G$. A locally constant cocycle 
$(\tilde{A},\tilde{f})$ 
is equivalent to $(A,f)$ if and only if there is $g \in G$ such 
that $g \ \chi_{(A,f)}=\chi_{(\tilde{A},\tilde{f})}\ g$.
\end{proof}

\begin{remark}
\label{rem_h1lc}
Let $G$ be a {\em topological} group. We consider the set $\rZ^1(M,G)$ 
of {\em cocycles} 
in the sense of \cite[Chp.I]{Kar}, 
and the associated cohomology $\rH^1 ( M,G )$. We recall 
that elements of $\rZ^1(M,G)$ are pairs of the type $(A,f)$, where $A$
is an open, locally finite cover of $M$, and $f$ is a family of 
{\em continuous} maps 
$f_{aa'} : a \cap a' \to$ $G$, $a,a' \in A$, satisfying the cocycle relations. It is 
well-known that such cocycles are in one-to-one correspondence with transition maps of 
principal bundles over $M$. The natural map
\begin{equation}
\label{def_lc_c}
\rH_{lc}^1 (M,G) \to \rH^1(M,G)
\end{equation}
{\em is generally not injective}: in fact, inequivalent cocycles in $\rH_{lc}^1 (M,G)$
may become equivalent in $\rH^1(M,G)$.
\end{remark}

To make the isomorphisms of the previous theorem explicit, 
we provide an explicit map from $\rH_c^1 (K^\circ,G) \simeq \rH^1 (K,G)$ 
to $\rH_{lc}^1 (M,G)$:
\begin{proposition}
\label{prop_c_lcc} 
There is an isomorphism
$\rH_c^1 (K^\circ,G) \to \rH_{lc}^1 (M,G)$.
\end{proposition}

\begin{proof}
Let $A \subseteq K$ be a fixed locally finite, open cover, and $\xi$ a \v Cech 
cocycle. Then, we have a family of locally constant maps
\[
\xi_{aa'} : V^\circ_{aa'} \to G
\ \ , \ \
a,a' \in \tilde\Sigma_0(K^\circ)
\ .
\]
Now, since each $a \cap a'$ is an open set, we find 
$a \cap a' =$
$\cup_{o \in V^\circ_{aa'}} o$.
Since $\xi_{aa'}$ is locally constant on $V^\circ_{a'a}$, we can
define the following locally constant function on $a' \cap a$:
\begin{equation}
\label{eq_wau}
\xi^{\bf lc}_{aa'} (x) 
:= 
\xi_{aa'} (o) 
\ \ , \ \ 
o \in V^\circ_{aa'} , x \in o 
\ \ .
\end{equation}
The family $\xi^{\bf lc} :=$ 
$\left\{ \xi^{\bf lc}_{aa'} \right\}$ clearly satisfies the 
cocycle relations, so $( A , \xi^{\bf lc} )$ is a locally constant cocycle. 
If $\tilde{A} \subseteq K$ is 
another locally finite, open cover, then we obtain a locally constant cocycle 
$( \tilde{A} , \tilde{\xi}^{\bf lc} )$, 
with $\tilde{\xi}^{\bf lc}_{\tilde{a} {\tilde{a}'}} (x)$, $x \in \tilde{a} \cap \tilde{a}'$, 
$\tilde{a} , \tilde{a}' \in \tilde{A}$, defined as in (\ref{eq_wau}). By defining
$v_{a \tilde{a}} (x) :=$
$\xi_{a \tilde{a}} (o)$,
$o \in V^\circ_{a\tilde{a}}$, 
$x \in o$,
we find
$v_{a \tilde{a}} (x) 
\tilde{\xi}^{\bf lc}_{\tilde{a} {\tilde{a}'}} (x)  
v_{a' \tilde{a}'} (x)^{-1} =$
$\xi^{\bf lc}_{aa'} (x)$,
thus $(\tilde{A},\tilde{\xi}^{\bf lc})$ is equivalent to $(A,\xi^{\bf lc})$. 
This implies that the equivalence class 
of $(A,\xi^{\bf lc})$ in $\rH_{lc}^1 ( M,G )$ does not depend on the choice 
of $A$.
If $\xi'$ is equivalent to $\xi$, then  $( A ,\xi^{\bf lc} )$ 
is equivalent to $( A , {\xi'}^{\bf lc} )$ for every open, locally finite, cover
$A$; thus, the map 
\begin{equation}
\label{eq_c_lc}
\xi \mapsto (A,\xi^{\bf lc})
\in \rH_{lc}^1 (M,G)
\ \ , \ \
\xi \in \rH_c^1 (K^\circ,G)
\ ,
\end{equation}
is well defined at the level of cohomology classes.
We prove the injectivity of (\ref{eq_c_lc}). To this end, note
that if $(A,\xi^{\bf lc})$ $(A,{\xi'}^{\bf lc})$ are equivalent, then the above
argument shows that  $( \tilde{A},\xi^{\bf lc} )$ and $( \tilde{A},{\xi'}^{\bf lc})$
are equivalent for every locally finite open cover $\tilde{A} \subseteq K$. 
Hence $\xi$ and $\xi'$ are equivalent.
Finally, we prove the surjectivity of (\ref{eq_c_lc}).
If $( A,f )$ is a locally constant cocycle, consider the
associated representation $\chi : \pi_1(K) \to G$, and the
corresponding \v Cech cocycle $\xi \in \rH_c^1(K,G)$, see
Proposition \ref{E:4.0}, Theorem \ref{thm_hmor}. In particular, the set
$\xi_{aa'}$, $a,a' \in A$, is a locally constant cocycle $( A,\xi
)$, equivalent to $(A,f)$ by construction.
\end{proof}

The above constructions  can be summarized in the diagram
\begin{equation}
\rH_c^1(K^\circ,G) \stackrel{\simeq}{\longrightarrow}
 \rH_{lc}^1 (M,G) \longrightarrow
  \rH^1 (M,G)
\end{equation}
%
The first map (that emphasized by the symbol  "$\simeq$") 
is a isomorphism. The other is, in general, neither
injective nor surjective. 
For example, when $M$ is the $1$-sphere $S^1$ and
$G$ the torus $\mathbb{T}$, we find $\mathbb{T} \simeq$ 
$\Hom ( \mathbb{Z} , \mathbb{T}) \simeq$ $\rH^1(K,\mathbb{T}) \simeq$
$\rH_{lc}^1 ( S^1 , \mathbb{T})$, whilst $\rH^1 ( S^1 , \mathbb{T})$ is trivial.

\section{Conclusions} 

The results of this paper demonstrate how the basic concepts 
and results of the theory of fibre bundles admit an analogue 
for net bundles over posets. In a sequel to this paper, we will 
show that this continues to be true for the K-theory of an 
Hermitian net bundle over a poset, even if the results diverge 
more in this case. In particular, we shall define the Chern 
classes for such bundles.\smallskip 

Our current research aims to develop a far-reaching but very 
natural generalization of the notion of net bundle. This involves 
replacing the poset, the base category of our net bundle, and 
the fibre category, here usually a group, by two arbitrary 
categories, a generalization leading us in the direction of 
fibred categories \cite{Gro}.\smallskip

\appendix

\section{Connection 1--cochains over posets}
\label{Z}
This appendix is intended to provide the reader with the notation 
and a very brief outline  of the results of the cohomological 
description of connections over posets developed in
\cite{RR06}.\bigskip

The present paper treats principal bundles over posets 
having an arbitrary structure group. The lack of a differential 
structure forces us to use a cohomology taking values in the structure 
group $G$. Hence, in general, we will deal with a non-Abelian  
cohomology of posets. In \cite{RR06} the coefficients for the 
non-Abelian $n^{th}$-degree cohomology  of a poset $K$ are an $n$-category
associated with the group $G$. 
In the cited paper the set of $n$--cochains $\rC^n(K,G)$ is defined only  
for  $n=0,1,2,3$. With this restriction, 
there is a coboundary operator $\mathrm{d}:\rC^n(K,G)\to\rC^{n+1}(K,G)$
satisfying  the equation $\mathrm{d}\mathrm{d}=\imath$ where $\imath$ denotes 
the \emph{trivial} cochain. So the sets of \emph{$n$--cocycles} and
\emph{$n$--coboundaries} are defined, respectively, by 
\[
 \rZ^n(K,G):= \rC^n(K,G)\cap Ker(\mathrm{d}), \ \ \ 
 \rB^n(K,G):= \rC^n(K,G)\cap Im(\mathrm{d}).
\]
Here we describe the category of 1--cochains 
in some detail, this being all we need.\smallskip

For $n=0,1$, an  \emph{$n$--cochain} is just a 
mapping $v:\tilde\Si_n(K)\to G$.  Given a $1$--cochain 
$v\in\rC^1(K,G)$, we can and will extend $v$ 
from $1$--simplices to paths by defining for $p=b_n*\cdots*b_1$ 
\[
v(p)  := v(b_n)\, \cdots\,  v(b_2)\, v(b_1).
\]
Given  $v,\tilde v\in\rC^1(K,G)$,
a \emph{morphism} $\mathrm{f}$ from $\tilde v$ to $v$ 
is a function $\mathrm{f}:\tilde\Si_0(K)\rightarrow G$  satisfying the 
equation 
\[
  \mathrm{f}_{\partial_0p} \,  \tilde v(p)  = v(p) \, 
\mathrm{f}_{\partial_1p}, 
\]
for all paths $p$. We denote the set of morphisms from 
$v_1$ to $v$ by $(\tilde v,v)$. 
There is an obvious composition law between morphisms given by pointwise 
multiplication making  $\rC^1(K,G)$ into a category. 
The identity arrow $1_v\in(v,v)$ takes the constant value 
$e$, the identity of the group. 
Given a group homomorphism $\gamma:H\to G$ and a 
morphism $\mathrm{f}\in(\tilde v,v)$ of 
$1$--cochains with values in $H$ then $\gamma\circ v$, defined by 
\begin{equation}
\label{Z:1}
(\gamma\circ  v) (b) := \ga(v(b)), \qquad b\in\tilde\Si_1(K), 
\end{equation}
is a $1$--cochain with values in $G$, and $\gamma\circ\mathrm{f}$
defined by 
\begin{equation}
\label{Z:2}
(\gamma\circ\mathrm{f})_a := \ga(\mathrm{f}_a), \qquad a\in\tilde\Si_0(K), 
\end{equation}
is a morphism of $(\gamma\circ \tilde v,\gamma\circ v)$. 
One checks at once that $\gamma\circ$ is a functor from $\rC^1(K,H)$ to 
$\rC^1(K,G)$, and that if $\ga$ is a group isomorphism, then 
$\ga\circ$ is an isomorphism of categories.\smallskip 

Note that $\mathrm{f}\in(v_1,v)$ implies $\mathrm{f}^{-1}\in(v,v_1)$, where
$\mathrm{f}^{-1}$ here denotes the composition of $\mathrm{f}$ with the inverse of 
$G$. We say that $v_1$ and $v$ are \emph{equivalent}, 
written $v_1\cong v$, whenever $(v_1,v)$ is nonempty. 
Observe that  a $1$--cochain $v$ is equivalent 
to the trivial $1$--cochain $\imath$ if, and only if, it is 
a $1$--coboundary. We will say that $v\in \rC^1(K,G)$  is 
\emph{reducible} if there exists a proper subgroup 
$H\subset G$ and a $1$--cochain $\tilde v\in\rC^1(K,H)$ with  
$\iota_{\sst{G,H}}\circ \tilde v$ equivalent to $v$, where 
$\iota_{\sst{G,H}}$ denotes the 
inclusion $H\subset G$. If $v$ is not reducible it will be said 
to be \emph{irreducible}.\smallskip

A \emph{$1$--cocycle} of $K$ with values in $G$  is a 1--cochain $z$ 
satisfying the equation 
\begin{equation}
\label{Z:3}
 z(\partial_0c)\, z(\partial_2c) = z(\partial_1c), \qquad 
 c\in\tilde\Si_2(K). 
\end{equation}
The \emph{category of $1$--cocycles} with 
values in $G$, is the full subcategory of $\rC^1(K,G)$ whose set of 
objects is $\rZ^1(K,G)$. We denote this category 
by the same symbol  $\rZ^1(K,G)$ as used to denote the corresponding 
set of objects. Clearly, 1--cohomology 
is strictly related to the first homotopy group.  
One first observes that any 
1--cocycle $z$ is \emph{homotopic invariant}, i.e.,
$z(p)=z(q)$ whenever $p$ and $q$ are homotopic paths. 
Using this property,  
\begin{equation}
\label{Z:4}
\rZ^1(K,G)\cong  H(\pi_1(K,a),G), 
\end{equation}
that is, the category $\rZ^1(K,G)$ \emph{is equivalent} to the category 
$H(\pi_1(K,a),G)$
of group homomorphisms from $\pi_1(K,a)$ into $G$. 
Hence,  if $K$ is simply 
connected, then any 1--cocycle is a 1--coboundary.
The set of \emph{connections} with values in $G$ is the subset 
$\rU^1(K,G)$ of those 1--cochains $u$ of $\rC^1(K,G)$ 
satisfying the properties 
\begin{equation}
\label{Z:5}
\begin{array}{lcl}
 (i)  & u(\overline{b})= u(b)^{-1}, &   b\in\tilde\Sigma_1(K), \\[3pt] 
 (ii) & u(\partial_0c)\,  u(\partial_2c) = u(\partial_1c), & 
    c\in  \Si_2(K).
\end{array} 
\end{equation}
The \emph{category of connection $1$--cochains} with values in $G$, is 
the full subcategory of $\rC^1(K,G)$ whose set of objects 
is $\rU^1(K,G)$.  It  is  denoted 
by  $\rU^1(K,G)$ just as  the corresponding 
set of objects. \smallskip

The interpretation of 1--cocycles of $\rZ^1(K,G)$ 
as principal bundles over $K$ with structure group $G$ derives from the 
following facts. Any connection $u$ of $\rU^1(K,G)$ 
induces a unique $1$--cocycle $z\in\rZ^1(K,G)$ 
satisfying the equation 
\begin{equation}
\label{Z:6}
u(b) = z(b), \qquad b\in \Si_1(K).
\end{equation}
$z$ is called the cocycle \emph{induced} by $u$. Denoting 
the set of connections of inducing 
the $1$--cocycle $z$ by  $\rU^1(K,z)$, we have that 
\begin{equation}
\label{Z:7}
 \rU^1(K,G) = \dot{\cup}\big\{ \rU^1(K,z) \ | \
 z\in\rZ^1(K,G)\big\},
\end{equation}
where the symbol $\dot{\cup}$ means disjoint union. 
So, the set  $\rU^1(K,z)$ can be seen as the set of connections 
of the principal bundle associated with $z$.  
We call the category of
\emph{connections inducing} $z$, the full subcategory 
of $\rU^1(K,G)$ whose objects belong to $\rU^1(K,z)$, and 
denote this category 
by $\rU^1(K,z)$ just as  the corresponding
set of objects.\smallskip 

The relation between 
cocycles (connections) taking values in different groups is easily established.
Given a group homomorphism $\ga:H\to G$,
then the restriction of the functor 
$\ga\circ$  to $\rZ^1(K,H)$  ( $\rU^1(K,H)$ ) defines 
a functor from $\rZ^1(K,H)$ into $\rZ^1(K,G)$ and 
( $\rU^1(K,H)$ into $\rU^1(K,G)$ ). This functor is an isomorphism 
when $\ga$ is a group isomorphism.

\end{document}